\documentclass[a4paper,10pt]{article}

\usepackage{amsmath}
\usepackage{amsthm}

\usepackage{epsfig}
\usepackage{epstopdf}
\usepackage{amssymb}
\usepackage{amsmath,amsfonts}
\usepackage{graphicx}
\usepackage{epstopdf}
\usepackage{upgreek}
\usepackage{color}

\usepackage{soul} 

\usepackage[margin=3cm]{geometry}

\title{Skew Brownian diffusions across Koch interfaces}

\author{\textsc{Raffaela Capitanelli}\footnote{
Dipartimento di Scienze di Base e Applicate per l'Ingegneria, Sapienza"  Universit\`{a}  di Roma,
Via A. Scarpa 16,  00161 Roma, Italy} \ \
\& \ \
\textsc{Mirko D'Ovidio}$^\ast$\\
{\small raffaela.capitanelli@uniroma1.it, mirko.dovidio@uniroma1.it}
}

\newtheorem{theorem}{Theorem}[section]

\newtheorem{proposition}{Proposition}[section]

\newtheorem{remark}{Remark}[section]

\newtheorem{definition}{Definition}[section]

\newcommand{\N}{{\mathbb N}}
\newcommand{\C}{{\mathbb C}}

\def\phi{\varphi}
\def\div{\mathop{\rm div}}

\def\phi{\varphi}
\def\div{\mathop{\rm div}}

\def\phi{\varphi}

\newcommand{\re}{\mathbb{R}}

\def\beq{\begin{equation}}
\def\eeq{\end{equation}}
\def\beqn{\begin{eqnarray}}
\def\eeqn{\end{eqnarray}}

\numberwithin{equation}{section}
\numberwithin{figure}{section}

\allowdisplaybreaks

\begin{document}

\maketitle

\begin{abstract}
We consider planar skew Brownian motion (BM) across pre-fractal Koch interfaces $\partial \Omega^n$ and moving on $\overline{\Omega^n} \cup \Sigma^n= \Omega^n_\varepsilon$ where $\Sigma^n$ is a suitable neighbourhood of $\partial \Omega^n$. We study the asymptotic behaviour of the corresponding multiplicative functionals  when  thickness of $\Sigma^n$ and skewness coefficients vanish with different rates. Thus, we provide a probabilistic framework for studying diffusions across semi-permeable pre-fractal (and fractal) layers and the asymptotic analysis concerning the insulating fractal layer case.
\end{abstract}

 

\textbf{GRANT}: P.U. Sapienza Universit\`{a} di Roma 2014.

\section{Introduction}

\paragraph*{State of the Art.} Diffusions on irregular domains have been investigated by many authors as well as the construction of reflecting Brownian motions on non smooth domains (\cite{BasHsu91, ChenPTRF93, Fuk67, FUK-book}). However, if the domain $D$ is Lipschitz, then we can construct the usual reflecting BM as in \cite{BasHsu91}. Let $D \subset \mathbb{R}^d$, $d\geq 2$, a bounded Lipschitz domain. Existence and uniqueness of the solution to $dX_t = dB_t + \mathbf{n}(X_t) dL^{\partial D}_t$ have been investigated in \cite{BasBurECP06, BasBurChe05} when $\mathbf{n}(z)$ is the inward normal vector at $z \in \partial D$ and $L^{\partial D}_t$ is the local time of $X$ on the boundary of $D$. In particular, $L^{\partial D}_t$ is a non-decreasing process such that $\int_0^\infty \mathbf{1}_{D}(X_t)dL^{\partial D}_t = 0$ that is, the process does not increase inside $D$. The local time can be associated with the surface measure (\cite{BBC, BasHsu91}) in the sense of the  Revuz correspondence. Moreover, convergence of reflecting BM in varying domain has been also investigated (see for example \cite{BurdzyChenECP98} and the references therein). In \cite{BBC} the authors studied the Robin problem on fractal domains in the framework of the so called \emph{trap domains} (see \cite{BCM}) which is a nice property to deal with for our purposes. We also deal with processes which are skew diffusions.  The skew BM has been introduced in \cite{HShep81, itoMck74, walsh78} and constructed to model permeable barrier in \cite{port79a, port79b}. An interesting surveys can be found in \cite{Lejay06}. It has been also investigated by many researchers as a tool in applied sciences. Applications to a single interface have been developed in \cite{ABTWW11, HShep81, itoMck74, lejMart06, oukn90, ramirez11, walsh78}. Recent results on multidimensional skew BM can be found in \cite{AtarBudSPA, trutnau03, trutnau05}. In \cite{lang95,weinryb84} the authors approach homogenization problems. As well described in \cite[pag. 272]{itoMck74}, it is possible to construct a reflecting BM $B^+$ on $\Omega \subset D$, a subset of $\mathbb{R}^2$, by considering a BM $B$ on $D$ and the occupation time $\mathfrak{f}$ of $B$ on $\Omega$. That is, $B(\mathfrak{f}^{-1})$ is identical in law to $B^+$.  It is also shown in \cite{itoMck74} that by killing $B(\mathfrak{f}^{-1})$ at a random time $T$ with conditional law $\mathbb{P}(T >t | B(\mathfrak{f}^{-1})) = \exp - \int \ell(\mathfrak{f}^{-1}(t), x) \kappa(dx)$, one obtains the connection with the motion driven by the Feynman-Kac generator ($\ell$ is a local time and $\kappa$ is a killing rate). An interesting connection has been also given by verifying a conjecture of Feller. Indeed, an elastic BM on $[0, \infty)$ with elastic condition $\gamma u(0) = (1-\gamma) u^\prime(0)$, $\gamma \in (0,1)$ is identical in law to $B^+$ killed according with the conditional law $\mathbb{P}(T >t | B^+) = \exp - \frac{\gamma}{1-\gamma} \ell^+(t,0)$.  We notice that the special cases $\gamma=1$ or $\gamma=0$ correspond to Dirichlet or Neumann conditions.
\paragraph*{Our results.} In this paper we consider boundary value problems on snowflake domain $\Omega$ by using the homogenization results obtained in \cite{CV0, CV-asy} with the approach of insulating layers (see, for example, \cite{AB, BCF} in smooth layers). More precisely, the fractal layer is approximated by a two-dimensional insulating thin layer $\Sigma^n$ with vanishing thickness and decreasing conductivity. Therefore, the emerging operators have discontinuous coefficients on the pre-fractal interfaces $\partial \Omega^n$ and so we consider skew Brownian motions, that is generalized diffusions processes (see, for example,  \cite{port79a,port79b} and \cite{trutnau05}). 

More precisely, the process we are dealing with is a skew planar BM on a bounded domain $\Omega^n_\varepsilon = \overline{\Omega^n} \cup \Sigma^n$ with pre-fractal interface $\partial \Omega^n$. We say that the BM in $\Omega^n_\varepsilon$ is skew meaning that it has different probability to stay in either $\overline{\Omega^n}$ or $\Omega^n_\varepsilon \setminus \Omega^n$. We have a skewness condition on the boundary $\partial \Omega^n$.  We denote by $B^{\nu,*}_t$ the skew (modified) planar BM on $\Omega^n_\varepsilon$ and we focus on the multiplicative functional $M^n_t = \mathbf{1}_{(t < {\zeta^{\Omega^n_\varepsilon}})}$ of $B^{\nu,*}_t$  where $\zeta^{\Omega^n_\varepsilon}$ is the lifetime of $B^{\nu, *}$ on $\Omega^n_\varepsilon$ and $\nu$ is the skewness parameter (see Section \ref{sec-irregular}).

In our analysis, we mainly focus on occupation measures and stopping times. A key role is played by the fact that the pre-fractal and fractal Koch domains are non trap. Thus, the fact that the semi-permeable barrier is given by the pre-fractal curve $\partial \Omega^n$ does not affect our discussion in terms of occupation measures. Let $T$ be the lifetime of the skew Brownian motion $B^{\nu, *}$ and $c_n$ be a sequence of positive constants describing the transmission condition on $\partial \Omega^n$. Under the non-restrictive assumption that $T=T_{c_n}$ (that is the lifetime depends on $c_n$) we consider the lifetime $\widehat{T_{c_n}}$  with conditional law $\mathbb{P} (\widehat{T_{c_n}}> t | B^{\nu, *}) = \exp - c_n \sigma_n \int_{\partial \Omega^n} \ell^{n}_t(y) d\mathfrak{s}$ (see Section \ref{subsection:localtime}) where $\sigma_n$ is a structural constant associated with the arc-length measure $\mathfrak{s}$ on the pre-fractal boundary. In particular, we consider a sequence of exponential random variables $\zeta^n$ with parameter $c_n \in [0, \infty]$ from which we construct a sequence of stopping times $\widehat{\zeta^{\Omega^n}}$ (see formula \eqref{def-life-sequence} below) depending on the time the process spends on (or cross) the pre-fractal interfaces. 

 Our aim is to investigate the asymptotic behaviour of $M^n_t$ when  thickness (of $\Sigma^n$) and skewness coefficients vanish with different rates according with $c_n$. We show that the limit process can be the elastic, reflecting or absorbing Brownian motion according to the asymptotic behaviour of the parameter $c_n$ (see Theorem \ref{thm-mu}). Our approach is based on the study of the asymptotic behaviour of $ \mathbf{1}_{(t < \widehat{\zeta^{\Omega^n}})}$ or equivalently $\widehat{M^n_t} = \mathbf{1}_{(t < \widehat{T_{c_n}})}$. 

Concerning the Dirichlet problem on $D\subset \mathbb{R}^d$, the connection between variational and probabilistic approach to diffusion equations with killing has been investigated for example in \cite{BaxterDalMasoMosco}. Boundary value problems with varying domains has been also investigated in \cite{ButtazzoDMasoMosco, Stollmann} where a key role is played by the capacity induced by a regular Dirichlet form.  

\paragraph*{Plan of the work.} The plan of the paper is the following. Section \ref{sec-Notation} introduces notation and definitions of the pre-fractal and fractal Koch curves. Moreover, we recall the homogenization results obtained in \cite{CV-asy}.  Section \ref{sec-PCAF} gives some basic aspects about positive continuous additive functionals and random times. In Section \ref{sec-regular} we consider skew BM across a regular layer. The skew BM across irregular boundaries is introduced in Section \ref{sec-irregular}. Our main results are collected and discussed in Section \ref{sec-results}.

\section{Notation and preliminary results}
\label{sec-Notation}

In this section we introduce the notation and  some preliminary results. We recall the definition of the Koch curve with endpoints $A = (0, 0),$ and $B = (1, 0)$. We consider the family $\Psi^\alpha=\{\psi^\alpha_1,\dots, \psi^\alpha_4\}$ of contractive similitudes  $\psi^\alpha_i:\C \rightarrow\C$, $i=1,\ldots,4$, with contraction factor $\alpha^{-1},$  $2<\alpha<4$, 
$$\aligned&\psi^\alpha_1(z)=\frac{z}{\alpha},\,\quad\quad\quad\quad\quad\quad\quad\quad\quad\quad\quad\quad
\psi^\alpha_2(z)=\frac{z}{\alpha}e^{i\theta(\alpha)}+\frac{1}{\alpha},\\
&\psi^\alpha_3(z)=\frac{z}{\alpha}e^{-
i\theta(\alpha)}+\frac12+i\sqrt{\frac{1}{\alpha}-\frac{1}{4}},\quad\quad\quad
\psi^\alpha_4(z)=\frac{z-1}{\alpha}+1,\endaligned$$
where $\theta(\alpha)=\arcsin\left( 
\frac{\sqrt{\alpha(4-\alpha)}}{2}\right).$

By the general theory of self-similar fractals (see \cite{Fal}), there exists a unique closed bounded set $K_\alpha$ which is {\it invariant\/} with respect to $\Psi^\alpha$, that is, 
\begin{equation}
K_\alpha=\cup^4_{i=1}\psi^\alpha_i(K_\alpha).
\end{equation}
We recall that $K_\alpha$ supports a unique self-similar Borel measure 
\begin{align}
\label{mu-alpha-def}
\mu_\alpha \; \textrm{which is equivalent to the $d_f-$dimensional Hausdorff measure}
\end{align}
where $d_f=\frac{\log 4}{\log \alpha}$. Let $K^0$ be the line segment of unit length that has as endpoints $A=(0,0)$ and  $B=(1,0)$. We set, for each $n$ in $\N$, 
\begin{equation}
\label{Kn}
 K_\alpha^1=\bigcup_{i=1}^4 \psi^\alpha_i(K^0),\qquad
 K_\alpha^2=\bigcup_{i=1}^4 \psi^\alpha_i(K_\alpha^1), \qquad
 \dots, \qquad
 K_\alpha^{n+1}=\bigcup_{i=1}^4 \psi^\alpha_i(K_\alpha^n);
\end{equation}
${K_\alpha^n}$ is the so-called ${n}${-th} {pre-fractal curve}. Moreover, the iterates $K_\alpha^{n}$ converge to the self-similar set $K_\alpha$ in the Hausdorff metric, when $n$ tends to infinity. Let $\Omega^0$ be the triangle  with vertices $A = (0, 0), B = (1, 0),$ and $C = (\frac12, -\frac{\sqrt{3}}{2})$. We construct on the side with endpoints $A$ and $B$ the pre-fractal Koch curve defined before, which will be denoted by  $K^n_{1,\alpha}$ and the Koch curve defined before, which will be denoted by $K_{1,\alpha}$. In a similar way, we construct on the  other sides  the analogous pre-fractal Koch curves (the Koch curves) denoting  by   $K^n_{2,\alpha}$  and $K^n_{3,\alpha}$ (by $K_{2,\alpha}$ and $K_{3,\alpha}$) the  curves with endpoints $B$ and $C$, and $C$ and $A$, respectively.  We denote by  $\Omega_\alpha^n$  the pre-fractal domain that is the set bounded by the pre-fractal Koch curves $K^n_{ j,_\alpha},$  $j=1,2,3.$ Moreover,  we denote by $\Omega_\alpha$  the snowflake that is the set  bounded by the Koch curves $K_{ j,_\alpha},$  $j=1,2,3$ (see Figure \ref{fig1}). We denote by  $\Sigma_1^0$ the \emph{open set condition} triangle of vertices  $A=(0,0)$, $B=(1,0)$ and $C=(1/2, b/2)$ where $b=\tan(\frac{\theta}{2})$.
\begin{figure}
\centering
\includegraphics[width=7cm]{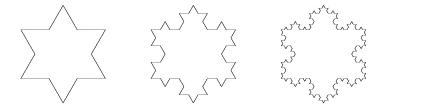}
\caption{The pre-fractal domains.}
\label{fig1}
\end{figure}

Following the construction in \cite{cap3}, for every $n$ and $\varepsilon$, we define the fiber  $\Sigma^{n}_{1,\alpha},$ \emph{$\varepsilon$-neighborhood} of $K^n_{1,\alpha}$  to be the (open) set
$$
\Sigma^{n}_{1,\alpha}=\bigcup_{i|n}\Sigma_{1,\alpha}^{i|n}, \quad \textrm{where} \quad \Sigma_{1,\alpha}^{i|n}=\psi^\alpha_{i | n}(\Sigma^0_1)$$
(see Figure \ref{fig3}). We proceed in a similar way in order to construct the fiber $\Sigma^{n}_{j,\alpha},$ \emph{$\varepsilon$-neighborhood} of $K^{n}_{j,\alpha}$ $(j=2,3)$ and, we define  the fiber $\Sigma^{n}_{\alpha}$,  \emph{$\varepsilon$-neighborhood} of $\partial \Omega^{n}$,   
$$\Sigma^{n}_\alpha=\bigcup^3_{j=1} \Sigma^{n}_{j,\alpha} \quad \textrm{and} \quad \Omega^{n}_{\varepsilon,\alpha}= \overline{\Omega^{n}_{\alpha}}\bigcup \Sigma^{n}_\alpha.$$ 
From now on, we omit  $\alpha$ when it does not give  rise   to misunderstanding, by writing simply  $\Omega$ instead of $\Omega_\alpha$ or $\mu$ instead of $\mu_\alpha$ and similar expressions. Moreover, we denote by $C$ positive, possibly different constants that do not depend on $n$  and on $\varepsilon$. We note that $$\Omega^{n}\subset\Omega^{n+1}\subset \Omega\subset{\Omega^{n+1}_{\varepsilon}}\subset {\Omega^{n}_{\varepsilon}}.$$
We define a $\emph{weight}$ $w^n$ as follows. Let $P$ -- for some $i|n$ -- belong to the boundary $\partial(\Sigma_{1}^{i|n})$ of $\Sigma_{1,\varepsilon}^{i|n}$ and let $P^{\perp}$ be the orthogonal projection of $P$ on $K_1^{i|n}$. If $x \in \mathbb{R}^2$ belongs to the segment with end-points $P$ and $P^{\perp}$, we set, in our current notation,
$$ w^n_{1}(x)= \frac{ 3|P-P^{\perp}| }{3+b^2}, $$ 
where $|P-P^{\perp}|$ is the (Euclidean) distance between $P$ and $P^{\perp}$ in $\re^2$. We proceed in a similar way in order to construct the weights $w^n_{j}$ on 
$\Sigma^{n}_{j}$ $(j=2,3)$
and we define $w^n$ on $\Omega^{n}_{\varepsilon}$
\begin{equation}\label{w(n)one}
w^n(x)= 
\begin{cases}
w^n_{j}(x)&\text{if}\quad x \in \Sigma^{n}_{j}\\
 1 &  \text{if}\quad x \in \overline{\Omega^{n}}.
 \end{cases}\end{equation}
Associated with the weight
$w^n,$ we consider the Sobolev spaces $H^1(\Omega^{n}_{\varepsilon};w^n)$
and $H^1_0(\Omega^{n}_{\varepsilon}; w^n)$, defined as  the 
completion of $C^{\infty}(\overline{\Omega^{n}_{\varepsilon}})$ and  $C^{\infty}_0(\Omega^{n}_{\varepsilon}),$ respectively,  in the norm 
\begin{equation}\label{H0}
\|u\|_{H^1(\Omega^{n}_{\varepsilon};w^n)}=\left( \int_{\Omega^{n}_{\varepsilon}} u^2 dx + \int_{\Omega^{n}_{\varepsilon}} |\nabla u|^2 w^n dx  \right)^\frac12
\end{equation}
where $dx$ denote the $2$-dimensional Lebesgue measure.
\begin{center}
\begin{figure}
\centering
\includegraphics[width=4cm]{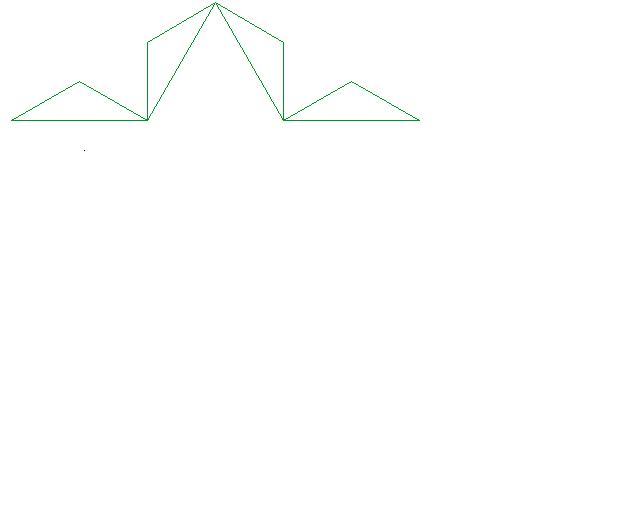}
\includegraphics[width=4cm]{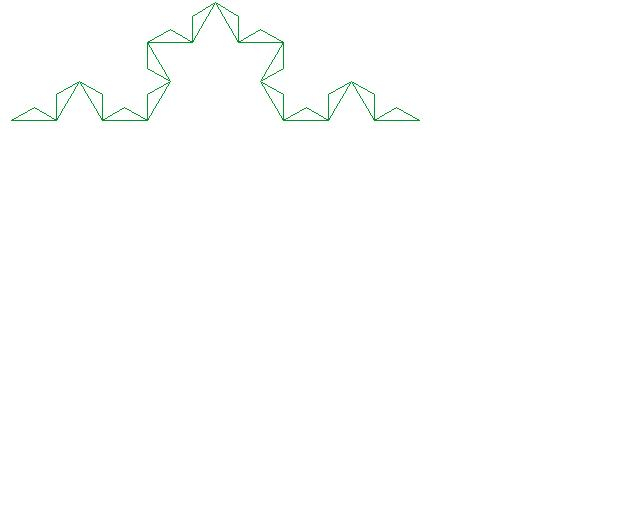}
\includegraphics[width=4cm]{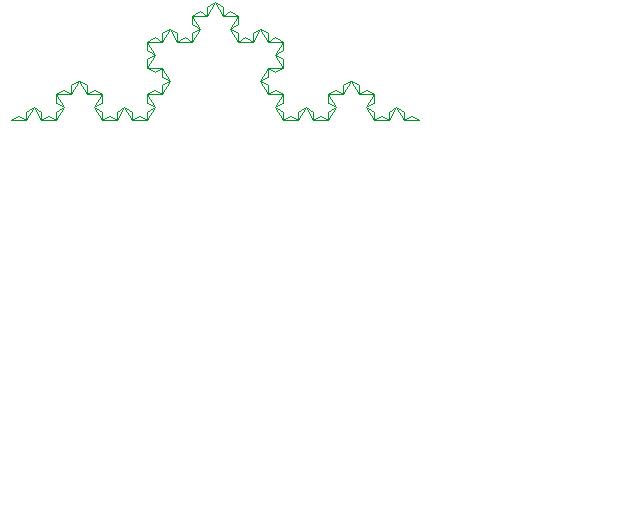}
\caption{The  fibers. }
\label{fig3}
\end{figure}
\end{center}
We define the  coefficients 
\begin{equation}\label{aq}
a^n_\varepsilon(x)= \begin{cases}
c_n\sigma_n \,w^n (x)& \text{if} \quad x \in \Sigma^{n}\\
1 & \text{if} \quad x \in\overline{\Omega^{n}},
\end{cases}
\end{equation}
where 
\begin{equation}\label{cn} 
c_n>0
\end{equation}
and
\begin{equation}\label{sq}
\sigma_n={\frac{\alpha^n}{4^{n}}}. 
\end{equation}
 
The following theorem states the existence and the uniqueness of  the  variational solution of the reinforcement problem. We consider the bilinear form associated  with the reinforcement problem 
\begin{equation}\label{Rina}
a_n(u,v):=\int_{{\Omega_{\varepsilon}^{n} }  } a^n_{\varepsilon}\nabla u\, \nabla v\, dx +\delta_n \int_{\Omega_{\varepsilon}^{n} } u\,  v\, dx 
\end{equation} 
where $a^n_{\varepsilon}$  is defined in \eqref{aq}, \eqref{sq}, \eqref{cn}, and $\delta_n>0.$

We denote by $\mathfrak{s}$ the arc-length measure on $\partial \Omega^n$.

\begin{theorem}
\label{pnm} 
Let $\sigma_n$ be as in \eqref{sq} and $d_n\in\re.$  Then,  for any $f_n\in L^2(\Omega_{\varepsilon}^{n}),$   there
exists one and only one solution $u_n$ of the following problem
\begin{equation}
\label{prefr.pb.varm}
\begin{cases}
\text{find}\;\; u_n\in 
H_0^1(\Omega_{\varepsilon}^{n}; w^n)\quad\text{such that}\\
a_n(u_n,v)=\int_{\Omega_{\varepsilon}^{n}} f_n\,v\,dx  +\sigma_n d_n\int_{\partial\Omega^{n}} \,v\,\, d\mathfrak{s}
\quad \forall \; v\in H_0^1(\Omega_{\varepsilon}^{n}; w^n),
\end{cases}
\end{equation} where $a_n(\cdot,\cdot)$  is defined in (\ref{Rina}).
Moreover, $u_n$ is   the only function that  realizes the minimum of the energy  functional
\begin{equation}\label{En}
\underset{v\in H_0^1(\Omega_{\varepsilon}^{n}; w^n)}{\min}\Big\{a_n(v,v)-2  \int_{\Omega_{\varepsilon}^{n} }  f_n\,v\,dx  -2\sigma_n d_n\int_{\partial\Omega^{n}} \,v\,\, d\mathfrak{s} \Big\}.
\end{equation}
\end{theorem}

In the following theorems, we state the  existence and  uniqueness of  the variational solution of the Robin, Neumann, and Dirichlet  problems on the  domain $\Omega$. We consider the bilinear form associated  with the Robin problem
\begin{equation}\label{Ra}
a_{c_0}(u,v):=\int_{\Omega} \nabla u\, \nabla v\, dx  +\delta_0 \int_{\Omega}
u\,  v\, dx  +c_0 \int_{\partial\Omega}
\gamma_0 u\,  \gamma_0 v\, d\mu
\end{equation}
where $\mu$ is the measure on $\partial \Omega$ that coincides, on each ${K_j}$ $j=1,2,3$, with the Hausdorff measure \eqref{mu-alpha-def} defined before  and $\gamma_0 u$ denotes the trace of  the function $u$ on the boundary of $\Omega$, that is for $v$ in $L^1_{loc}(D)$, where  $D$ is an arbitrary open set of $\re^2$,  the trace operator
$\gamma_0$ is defined as
\begin{equation}
\label{gamma0} \gamma_0 v(P):=\lim_{r\rightarrow
0}\frac{1}{m(B(P,r)\cap D)}\int_{B(P,r)\cap D}v(x)\,dx
\end{equation}
at every point $P\in \overline{D}$  where the limit exists  (see, for example,  page 15 in \cite{JonWal}). From now on,  we suppress $\gamma_0 $ in the notation,  when it does not give  rise to misunderstanding, by writing simply  $v$ instead of $\gamma_0 v$ and similar expressions. We assume that 
\begin{equation} \label{555}
c_0\geq 0,\, \delta_0 \geq 0,\, \text{and}\,\, \max(c_0,\delta_0)>0.
\end{equation}

\begin{theorem}
\label{th.1m} 
Let us assume \eqref{555} and $d\in \re.$ Then, for any $f\in L^2(\Omega),$ there exists one and only one solution $u$ of the following problem
\begin{equation}
\label{eq:5m}
\begin{cases}
 \text{find}\;\; u\in
H^1(\Omega)\quad\text{such that}\\
a_{c_0}(u,v) = \int_{\Omega} f\,v\,dx  +d \int_{\partial\Omega}
  v\, d\mu \quad\quad   \forall \; v\in H^1(\Omega) 
\end{cases}
\end{equation}
where $a_{c_0}(\cdot,\cdot)$ is defined in \eqref{Ra}. Moreover, $u$ is  the only function that  realizes the minimum of the energy  functional
 \begin{equation}\label{EE}
 \underset{v\in H^1(\Omega )}{\min}\Big\{a_{c_0}(v,v) -2  \int_{\Omega} f\,v\,dx 
-2 d \int_{\partial\Omega}  v\, d\mu\Big \}.
\end{equation}
\end{theorem}
In a similar way, we prove the following result. We consider the bilinear form associated with the Dirichlet problem and
\begin{equation}\label{RaD}
a_{\infty}(u,v):=\int_{\Omega} \nabla u\, \nabla v\, dx  +\delta_0 \int_{\Omega} u\,  v\, dx .
\end{equation}
We assume that 
\begin{equation} \label{5555} \delta_0 \geq 0.
\end{equation}

\begin{theorem}
\label{th.1Dm} Let us assume \eqref{5555}. Then, for any $f\in L^2(\Omega),$ there exists one and only one solution $u$ of the following problem
\begin{equation}
\label{eq:5mD}
\begin{cases}
 \text{find}\;\; u\in
H_0^1(\Omega)\quad\text{such that}\\
a_{\infty}(u,v) = \int_{\Omega} f\,v\,dx  \quad\quad   \forall \; v\in H_0^1(\Omega) 
\end{cases}
\end{equation}
where $a_{\infty}(\cdot,\cdot)$ is defined in \eqref{RaD}.
Moreover, $u$ is  the only function that  realizes the minimum of the energy  functional
 \begin{equation}\label{EED}\underset{v\in H_0^1(\Omega )}{\min}\Big\{a_{\infty}(v,v) -2  \int_{\Omega} f\,v\,dx \Big \}\\.
\end{equation}
\end{theorem}

We  recall the notion of $\emph{$M-$convergence}$ of functionals, introduced in \cite{MOS3}, (see also \cite{MOS1}).
\begin{definition}\label{def1}
A sequence of functionals $F^n: H \rightarrow (-\infty,+\infty]$ is said to $M-$converge to a functional $F: H \rightarrow (-\infty,+\infty]$ in a Hilbert space $H$, if
\begin{itemize}
\item[(a)] For every $u \in H$ there exists $u_n$
 converging strongly to $u$ in  $H$ such that
\beq\label{aa} \limsup F^n[u_n]\leq F[u], \quad as\quad
n\rightarrow + \infty.
 \eeq
\item[(b)] For every $v_n$ converging weakly to $u$ in $H$
\beq\label{bb} \liminf F^n[v_n]\geq F[u], \quad as\quad
n\rightarrow + \infty. 
\eeq
\end{itemize}
\end{definition}

Let $\Omega^*$ be an open regular domain  such that $\Omega^*\supset \overline{\Omega^{n}  _{\varepsilon}} ,$ for all $n:$  in order to fix  notation we choice  as  $\Omega^*$ the ball with the center in the point $P_0=(\frac{1}{2} ,- \frac{1}{2})$ and radius $1$. We consider the sequence of weighted energy functionals in $L^2(\Omega^*)$
\begin{equation}
F^n[u]=
\begin{cases}
\int_{\Omega^{n} _{\varepsilon}} a^n_\varepsilon |\nabla u|^2 dx + \delta_n \int_{\Omega^{n} _{\varepsilon}}  u^2 dx
 &\text{if} \, u|_{\Omega^{n} _{\varepsilon}}\in H^1_0(\Omega^{n}_{\varepsilon}; w^n)\\
+ \infty &\text{otherwise in }  L^2(\Omega^*)
\end{cases} \label{A(n)2}
\end{equation}
(the coefficients  $a^n_\varepsilon$ are defined in (\ref{aq}), (\ref {sq}), (\ref{cn}),  $\delta_n>0)$ and
\begin{equation}
F_{c_0}[u]=
\begin{cases}
\int_{\Omega}  |\nabla u|^2 dx + \delta_0 \int_{\Omega}  u^2 dx +c_0 \int_{\partial \Omega}
 u^2 d\mu&\text{if} \,u|_{\Omega} \in H^1(\Omega)\\
+ \infty  &\text{otherwise in }  L^2(\Omega^*).
\end{cases}
 \label{TeoA(u)22}
\end{equation}
Moreover, we consider the case where the layer is \emph{weakly insulating} (see \eqref{ISO} below) and we introduce the following  functional \eqref{Finf} in $L^2(\Omega^*)$  
\begin{eqnarray}\label{Finf}
F_\infty[u]=\left\{
\begin{array}{lll}
\int_{\Omega}  |\nabla u|^2 d x + \delta_0 \int_{\Omega}   u^2 d x  &\text{if} \quad u|_{\Omega} \in H_0^1(\Omega)\\
+ \infty \quad\quad\qquad\quad\quad\qquad\qquad &\text{otherwise in } L^2(\Omega^*). 
\end{array}
\right. 
\end{eqnarray}

In order to study the asymptotic behaviour of the  functions $u_n$, we fix the further assumptions
\begin{equation} \label{111}
f_n,\,f \in L^2(\Omega^*),\,\text{and} \,f_n\rightarrow f\,\,\in L^2(\Omega^*),\text{as}\,n\rightarrow +\infty,
\end{equation}

\begin{equation} \label{222}
\delta_n>0\,\,\text{and} \,\,\delta_n\rightarrow \delta_0\,\,\text{as}\,n\rightarrow +\infty,
\end{equation}

\begin{equation} \label{333}
c_n>0\,\,\text{and}\, \,c_n\rightarrow c_0\,\text{as}\,n\rightarrow +\infty,
\end{equation}

\begin{equation} \label{444}
d_n,d \in \re,\,\text{and} \,d_n\rightarrow d\,\text{as}\,n\rightarrow +\infty.
\end{equation}
We also introduce the following results which have been proved in \cite{CV-asy} and turn out to be useful further on.
\begin{proposition}\label{rr}  Let $\sigma_n$  be as in \eqref{sq}.
 Then, for every   sequence $g_n\in H^1({\Omega})$   weakly converging towards  $g^*$ in $H^1({\Omega}),$ we have 
\begin {equation}\label{cnH}
\sigma_n \int_{\partial \Omega^{n}} g_n d\mathfrak{s} \rightarrow \int_{\partial \Omega} g^*\,d\mu \,\,\,, \text{as}\,\,\,n \rightarrow +\infty.
\end {equation}
\end{proposition}

\begin{theorem}\label{teo1m} 
Let us assume \eqref{333} and \eqref{222}. Then, the sequence of functionals $F^n_{\varepsilon(n)}$, defined in (\ref{A(n)2}), $M-$converges in $L^2(\Omega^*)$ to the functional $F_{c_0}$  defined in \eqref{TeoA(u)22} as $n\rightarrow +\infty.$
\end{theorem}

Now we consider the case when the  conductivity of the thin fibers vanishes slower than the  thickness of the fiber: more precisely, we suppose
\begin{equation} \label{ISO}
c_n w^n\to 0,\,\, c_n\to + \infty. 
\end{equation} 

\begin{theorem}\label{teo2}  
Let us assume \eqref{ISO} and \eqref{222}.  Then the sequence of functionals $F^n_{\varepsilon(n)}$, defined in (\ref{A(n)2}), $M-$converges in $L^2(\Omega^*)$ as $n\rightarrow +\infty$ to the energy functional $F_\infty[u]$ defined in  (\ref{Finf}).
\end{theorem}

In conclusion, throughout we consider the geometric constant $\sigma_n$ as in Proposition \ref{rr} and the following condition on the conductivity of the thin fibers $\Sigma^n$ 
\begin{align}
\label{general-cw}
c_n w^n \to 0 \quad \textrm{as} \; n \to \infty.
\end{align}

\section{Positive continuous additive functionals and random times}
\label{sec-PCAF}

We recall some basic aspects and introduce some notations. Let $E$ be a locally compact separable metric space and $m$ be a positive Radon measure on $E$ such that $\text{supp} [m] = E$. A Dirichlet form $\mathcal{E}$  with domain  $D(\mathcal{E})$ is a Markovian closed  symmetric form on $L^2(E,m)$  (see \cite[Chapter 1]{FUK-book}). Let $X=(\{X_t\}_{t\geq 0}; \mathfrak{F}; \mathbb{P}_x, x \in E)$ be an $m$-symmetric Hunt process whose Dirichlet form $(\mathcal{E}, D(\mathcal{E}))$ on $L^2(E,m)$ is regular (see \cite[Chapter 5]{FUK-book}).

We say that $A_t$, $t\geq 0$ is a positive continuous additive functional (PCAF) and write $A_t \in \mathbf{A}^+_c$ denoting by  $\mathbf{A}^+_c$ the totality of PCAFs of an $m$-symmetric Hunt process $X$ (see \cite[A.3.1]{chen-book} for details). More precisely, we say that $A_t \in \mathbf{A}^+_c$ if 
\begin{itemize}
\item[A.1)] $A_t$, $t\geq 0$ is $\mathcal{F}_t$-measurable ($\{\mathcal{F}_t\}$ is the minimum completed admissible filtration),
\item[A.2)] there exists a set $\Lambda \in \mathcal{F}_\infty$ and an exceptional set $N \subset E$ with $\textrm{Cap}(N)=0$ such that $\mathbb{P}_x(\Lambda) = 1$ for all $x \in E \setminus N$, $\theta_t\Lambda \subset \Lambda$ for all $t>0$; for every $\omega \in \Lambda$, $A_t (\omega) : t \mapsto A_t(\omega)$ is continuous, $A_0(\omega)=0$; for all $s,t\geq 0$ $A_{s+t}(\omega) = A_t(\omega) + A_s(\theta_t \omega)$ where $\theta_t$, $t\geq0$ is the (time) translation semigroup,
\item[A.3)] for all $\omega \in \Lambda$, $A_t(\omega) : t \mapsto A_t(\omega)$ is non-decreasing.
\end{itemize}
In this section, we denote by $\mu$ a positive Radon measure on $E$. Hereafter, we write $\langle v,u \rangle_\mu = \int_E v(x) u(x) \mu(dx)$ and, in some case, we simply write $\langle v, \mu \rangle$ with obvious meaning of the notation. We denote by $C_0$ the set of continuous functions with compact support. A positive Radon measure $\mu$ for which (\cite[pag. 74]{FUK-book})
\begin{equation}
\int |v(x)| \mu(dx) \leq C \sqrt{\mathcal{E}_1(v,v)}, \quad \forall\, v \in D(\mathcal{E}) \cap C_0(E)  \label{smoothM}
\end{equation}
where
\begin{equation}
\mathcal{E}_\lambda(u,v) = \mathcal{E}(u,v) + \lambda \langle u,v\rangle_m \label{Elambda}
\end{equation}
is said of finite energy integral and formula \eqref{smoothM} holds if and only if there exists, for each $\lambda >0$, a unique function $U_\lambda \mu \in D(\mathcal{E})$ (where $U_\lambda \mu$ is a $\lambda$-potential) such that
\begin{equation}
\mathcal{E}_\lambda(U_\lambda \mu, v)=\int v(x)\mu(dx). \label{unique-mu}
\end{equation}  
We recall that (\cite[pag. 64]{FUK-book}), for an open set $B\subset  E$ and $\mathcal{L}_B=\{v \in D(\mathcal{E})\, :\, v\geq 1\, m \mbox{-} a.e.\ on\ B\}$, the capacity is defined as $\textrm{Cap}(B) = \inf_{u \in \mathcal{L}_B} \mathcal{E}_1(u,u)$ if $\mathcal{L}_B\neq \emptyset$ and  $\textrm{Cap}(B) =\infty$ if $\mathcal{L}_B= \emptyset$. We say that a Borel measure $\mu$ on $E$ is a smooth measure  and write $\mu \in S=S(E)$ if \cite[pag. 80]{FUK-book}
\begin{itemize}
\item[$\mu$.1)] $\mu$ charges no set of zero capacity;
\item[$\mu$.2)] there exists an increasing sequence $\{F_n\}$ of closed sets such that $\mu(F_n) < \infty$ and $\textrm{Cap}(K \setminus F_n) \to 0$ for all compact sets $K$.
\end{itemize}
The class of smooth measures $S$ is therefore large and it contains all positive Radon measures charging no set of zero capacity. 
By \cite[Lemma 2.2.3]{FUK-book}, all measures of finite energy are smooth. We use the notation introduced in \cite{FUK-book} and denote by $S_0 \subset S$ the set of positive Radon measure of finite energy integrals, by $S_{00} \subset S_0$ the set of finite measures with $\|U_1\mu\|_\infty < \infty$.  \\

Let us consider $\mu_A \in S$ and $A_t \in \mathbf{A}^+_c$ associated with the $m$-symmetric Hunt process $X$ with $\mathbb{P}_m(\Lambda) = \int_E \mathbb{P}_x(\Lambda)\, m(dx)$ and  $\mathbb{P}_x(\Lambda) = \mathbb{P}_x(X_t \in \Lambda)$ for $\Lambda \in \mathfrak{F}$. Then, the measure $\mu_A$ and the PCAF $A_t$ are in the Revuz correspondence if, for any $f \in \mathcal{B}_+(E)$ (the set of non-negative and measurable functions on $E$), we have that
\begin{align}
\langle f, \mu_A \rangle = \lim_{t \downarrow 0} \frac{1}{t} \mathbb{E}_m \left[ \int_0^t f(X_s)dA_s \right] =  \lim_{\lambda \to \infty} \lambda \mathbb{E}_m \left[ \int_0^\infty e^{-\lambda t} f(X_t)dA_t \right]. \label{Rev-Corr}
\end{align}  
We say that $\mu_A$ is the Revuz measure of $A \in \mathbf{A}^+_c$ and if $\mu_A \in S$, then there exists a unique (up to equivalence) PCAF $\{A_t\}_{t\geq 0}$ with Revuz measure $\mu_A$  (\cite[Theorem 5.1.4 and Theorem 5.1.3]{FUK-book}). Throughout, we write $\mu$ instead of $\mu_A$ if no confusion arises. Moreover, we introduce
\begin{equation}
\label{ReU-def}
R_\lambda f(x) = \mathbb{E}_x\left[ \int_0^\infty e^{-\lambda t} f(X_t) dt \right]
\quad
\textrm{and}
\quad
U_A^\lambda f(x) = \mathbb{E}_x\left[ \int_0^\infty e^{-\lambda t} f(X_t) dA_t \right]
\end{equation}
(see \cite{RevYor99} for a complete discussion).

We introduce some further notation and basic aspects. In the following sections we consider the killed process 
\begin{equation}
X_t = \left\lbrace 
\begin{array}{ll}
\widetilde{X_t}, & t< \tau\\
\partial , & t \geq \tau
\end{array}
\right . \label{first-X}
\end{equation} 
($\widetilde{X_t} \in E$ and $\partial$ is the \lq\lq coffin state" not in $E$)  where $\tau$ will be a suitable random time and $\mathbf{P}_t f(x) = \mathbb{E}_x[f(X_t)] = \mathbb{E}_x[f(\widetilde{X_t})\,;\, t < \tau]$,  $x \in E$ is the associated semigroup. In particular, we consider the following cases: i) $\tau= \zeta^{E}$ is a random time such that $(\zeta^{E} < t) \equiv (L^{\partial E}_t > \zeta)$ and $\zeta$ is an exponential random variable, with parameter $c_0 \in (0, \infty)$, independent from $X$; ii) $\tau=\infty$ under suitable conditions; iii) $\tau=\tau_E$ is the exit time of $X$ from $E$. 

Thus $X_t$, $t \in [0, \infty]$, is a Markov process with state space $E_\partial := E \cup \{\partial \}$. The transition function is not conservative according with the cemetery point $\{\partial\}$, that is $\mathbb{P}_x(X_t =\partial) \geq 0$, $\forall\, x \in E_\partial$, $t\geq 0$. In particular, $X$ is conservative if $\mathbb{P}_x(\zeta^E < \infty) =0$ for every $x \in E$ where we denote by $\zeta^E$ also the lifetime of the process on $E$. Since $X_t$ is a Markov process, $\mathbb{P}_x(X_0=x)=1$ for all $x \in E_\partial$ and $\mathbb{P}_\partial (X_t = \partial)= 1$ for all $t$. Our discussion is mainly concerned with trap domains. A point $x \in E_\partial$ is called a \emph{trap} of $X$ if $\mathbb{P}_x(X_t =x) =1$ for every $t\geq0$. We give the definition of trap domain further on in the text. In i) we have introduced the local time process $L^{\partial E}_t = L^{\partial E}_t(X)$ which is the PCAF increasing when $X$ hits the boundary $\partial E$. It is well known that, the lifetime of the process follows the law $\mathbb{P}_x(\zeta^{E} > t | X_t) = e^{-c_0 L^{\partial E}_t}$ for every $x \in E$ and $t>0$. Thus, $L^{\partial E}_0=0$ and $\mathbb{P}_x(\zeta^E >0) =1$. $L^{\partial E}$ is the occupation time of $X$ on $\partial E$. For $\Lambda \subseteq E$, we denote by $\Gamma^\Lambda_t(X) = meas\{ s \in [0,t]\,: \, X_s \in \Lambda \}$ the occupation time process of $X$ on $\Lambda$. The semigroup $\mathbf{P}_t$ is strongly continuous and we use the fact that $\lambda R_\lambda f \to f$ and $\lambda \langle  U_A^\lambda f, m\rangle \to \langle f, \mu \rangle$ as $\lambda \to \infty$ where $\mu$ is the Revuz measure associated with the additive functional $A$ and therefore, to the random time $\tau \in [\tau_E, \infty]$.  In particular, if $\mathbb{P}_x(\tau=\infty)=1$, then for the planar BM $B$, $\forall\, \Lambda \subseteq \mathbb{R}^2$, $\mathbb{P}_x$-almost surely, $\overline{\Gamma^\Lambda_t} (B) = \int_0^\infty e^{-\delta s} \mathbf{1}_\Lambda(B_s) ds = \infty$  if $\delta=0$.

Let us consider the perturbed Dirichlet form on $L^2(E, m)$ written as
\begin{equation}
\mathcal{E}^\mu_\lambda(u,v) = \mathcal{E}_\lambda(u,v) + \langle u, v \rangle_{\mu}, \quad u,v \in D(\mathcal{E}) \cap L^2(E, \mu) \label{D-Form-General}
\end{equation}
where $\mathcal{E}_\lambda$ has been introduced in \eqref{Elambda}, $\mu \in S$. Let $A_t \in \mathbf{A}^+_c$ and $\widetilde{X_t}$ as in \eqref{first-X}. The  transition function
\begin{equation}
\mathbf{P}_t^\mu f(x) = \mathbb{E}_x [e^{-A_t} f(\widetilde{X_t})] \label{semig-gen-A}
\end{equation}
is associated with the regular form $(\mathcal{E}^\mu_0, D(\mathcal{E}^\mu_0))$ where $\mu$ is the Revuz measure of $A_t$  (see \cite[Theorem 6.1.1 and Theorem 6.1.2]{FUK-book}). We simply write $\mathbf{P}_t$ instead of $\mathbf{P}_t^\mu$. In the following sections we consider $m$-version of $\widetilde{X}$ associated with our problems on fractal domains (and pre-fractal if clearly specified).

We say that $X^n$ converges in law to $X$ and write $X^n \stackrel{law}{\to}X$ if $\mathbb{E}f(X^n) \to \mathbb{E}f(X)$ as $n\to \infty$ for every continuous and bounded function $f$. 

Throughout, we consider the PCAF (in the strict sense, that is, in A.2) $\Lambda$ is the defining set and $N$ is an empty set) $A^n_t =  \int_0^t f(X^n_s)ds$, the multiplicative functional $M^n_t=e^{-A^n_t}$ and a stopping time $T_n$. We have that (see \cite[Lemma 2.1]{CFY2006})
\begin{align}
\lim_{t \downarrow 0} \mathbb{E}^n_x A^n_t = & \lim_{t \downarrow 0} \big[ \mathbb{E}^n_x[ A^n_t\,;\, t < T_n] + \mathbb{E}^n_x[ A^n_t\,;\, t \geq T_n] \big] =  \lim_{t \downarrow 0} \mathbb{E}^n_x[ A^n_t\,;\, t < T_n]. \label{mean-A-pcaf}
\end{align}

\section{Transmission condition on regular interfaces}
\label{sec-regular}

In this section we consider the probabilistic approach of thin layer when $\Omega$ is a disc. Actually,  we provide a sketch of proof for the problem with collapsing annulus by following two approaches. Here, the purpose is to underline the main differences with the fractal case investigated in the next sections. Notice also that speed measure and scale function characterize uniquely one-dimensional diffusions.

\paragraph*{First approach.}
Let us consider a BM $X$ on $\mathbb{R}^2$ started (at $x \in \mathbb{R}^2$) away from zero. For $\theta^\prime \in [0, 2\pi)$, $r^\prime>0$ we can write, $\mathbb{P}_x(X_t \in dy) = \mathbb{P}_{(\theta^\prime, r^\prime)}(\Theta_t \in d\theta, R_t \in dr)$ where $R=|X|$ is a Bessel process. In particular, $R$ and $\Theta$ are the radial and the angular part of $X$. It is also well-known that a skew-product representation is given in term of $(R, \Theta)$ where $R=|X|$ is a Bessel process and $\Theta= X/|X| = B(\int R^{-2}_z dz)$ with $B$ an independent BM on the sphere $\mathbb{S}^{1}$ \cite[pag. 269]{itoMck74}. Here $\Theta$ is a time-changed BM on $\mathbb{S}^{1}$. 

Let $\nu \in (0,1)$ and $B^\nu$ be a planar BM  on the disc $C_2$ with a disc $C_1 \subset C_2$ (centred at the same point $(0,0)$, with radius $r_1 < r_2=r_1+\varepsilon$, $\varepsilon>0$), Dirichlet condition on $\partial C_2$ and transmission condition on $\partial C_1$ (the skew condition, that is $\mathbb{P}_x(B^\nu_t \in C_2 \setminus \overline{C_1})=\nu$, $\mathbb{P}_x(B^\nu_t \in C_1)=1-\nu$ for $x \in \partial C_1$). Due to the non-symmetry $(1-\nu, \nu)$ we say that $B^\nu$ is a skew planar BM (that is a 2-dimensional extension of the skew BM, see for example Section 11.10 of \cite{Lejay06} or \cite{trutnau03}).  Let $\mathcal{L}_n$ be the governing operator of $B^\nu$.  We examine in this section the classical case corresponding to the (formal) problem
\begin{eqnarray*}
\displaystyle \mathcal{L}_n u_n  & = & -f_n  \quad \textrm{on } C_2\\
\displaystyle  (1-\nu) \, \partial_\mathbf{n} u_n \big|_{\partial C_1 -} & = & \nu\, \partial_\mathbf{n} u_n \big|_{\partial C_1 +} \\
\displaystyle u_n \big|_{\partial C_1 -} & = & u_n \big|_{\partial C_1 +} \\
\displaystyle u_n \big|_{\partial C_2} & = & 0
\end{eqnarray*}
where $\partial_\mathbf{n}$ is the normal derivative and we denote by $\partial C_1 -$ and $\partial C_1 +$ the boundary from the interior and from the exterior of $C_1$. Let us consider the sequences $\nu=\nu(n)$, $\varepsilon=\varepsilon(n)$, $n \in \mathbb{N}$. Our aim is to study the asymptotic behaviour of the solution as $n\to \infty$ and $\nu, \varepsilon \to 0$, $\nu/\varepsilon \to c$ with different rate given by the elastic coefficient $c \geq 0$. Then, the problem above can be associated with $B^\nu$ started away from the origin, that is he process is partially (normally) reflected on $\partial C_1$ and totally absorbed in $\partial C_2$. 

A reflecting BM on a disc can be constructed (in law) by considering suitable time change and rotation (\cite[pag. 272]{itoMck74}). The time change in this case is a stochastic clock given by an additive functional of the radial motion as indicated before. Denote by $C_{1,2}$ the annulus $C_2 \setminus \overline{C_1}$. Thus, for $0<r^\prime <r_2$ and $0<\theta^\prime  \leq 2\pi$, $\mathbb{P}_x(B^\nu_t \in dy) = \mathbb{P}_{(\theta^\prime, r^\prime)}(\Theta_t \in d\theta, R_t^\nu \in dr)$ where $R^\nu$ is a skew Bessel process on $(0, r_2)$ such that, $\mathbb{P}_{r_1}(R_t^\nu \in (r_1, r_2)) = \nu$. We have normal reflection on $\partial C_1$ and 
\begin{align}
\forall\, x \in \partial C_1, \qquad \mathbb{P}_x(B^\nu_t \in dy) = \frac{d\theta}{2\pi} \mathbb{P}_{r_1}(R_t^\nu \in dr). \label{unif-2pi}
\end{align}
The BM can move from $\partial C_1$ according with an uniformly distributed angle $\Theta$ for the choice of the starting point, that is
\begin{align}
\int_0^{2\pi} \mathbb{P}_{(\theta, r_1)}(B^\nu_t \in C_{1,2}) \frac{d\theta}{2\pi}= \nu \quad \textrm{and} \quad  \int_0^{2\pi} \mathbb{P}_{(\theta, r_1)}(B^\nu_t \in C_1) \frac{d\theta}{2\pi} = 1-\nu. \label{unif-2pi-2}
\end{align}

Let $R$ be the part of the Bessel process $\widetilde{R}$ on $(0, r_2)$ with $\widetilde{R} \in (0, \infty)$. We cut the excursions of $\widetilde{R}$ by considering a time change given by the inverse of $\Gamma_t^{(r_1, r_2)}(\widetilde{R})$. We do the same with $\Gamma_t^{(0,r_1)}(\widetilde{R})$. As in \cite[pag. 115]{itoMck74} we can obtain a skew motion by considering the $\nu$ portion of $\Gamma_t^{(r_1, r_2)}(\widetilde{R})$ and the $1-\nu$ portion of $\Gamma_t^{(0,r_1)}(\widetilde{R})$, that is a new occupation time, say $\mathfrak{f}$. Thus, it is possible to consider a suitable time change $\mathfrak{f}^{-1}$, in order to obtain partial (normal) reflection on $r_1$ and,  $R^\nu=R_{\mathfrak{f}^{-1}}$ is a Bessel process on $(0, r_2)$ with transmission condition on $r_1$. The skew BM constructed in this way has the skew-product representation involving the time-changed Bessel process $R_{\mathfrak{f}^{-1}}$ where the BM on the circle is identical in law to the original process (that is, $\Theta \stackrel{law}{=} \Theta_{\mathfrak{f}^{-1}}$). More precisely, let us consider $T^\nu_t = \int_0^{\mathfrak{f}^{-1}(t)} (R^\nu_z)^{-2} dz$ where $R^\nu_t = R_{\mathfrak{f}^{-1}(t)}$ and $T_t=\int_0^t R^{-2}_z dz$. Then $\Theta^\nu_t = B(T^\nu_t)$ where $B$ is independent from $T^\nu_t$ and $\Theta_t = B(T_t)$ where $B$ is independent from $T_t$. Since $T^\nu_t \stackrel{law}{=} T_t$ we get that $\Theta^\nu_t \stackrel{law}{=} \Theta_t$.

Thus, the only process we consider is the radial part $R_t$ time-changed by $\mathfrak{f}^{-1}$, that is $R^\nu$. The Bessel process can start from zero and then it is instantaneously reflected. It never hits the origin   at some $t>0$. The mean exit time 
\begin{align*}
v_\varepsilon(r) = \mathbb{E}[\tau_{(0, r_2)}(R^\nu) | R^\nu_0 =r \in (0,r_2)] = \mathbb{E}_r \tau_{C_2}
\end{align*}  
can be explicitly written by following standard techniques for one-dimensional diffusions (see for example \cite{KarlinTaylor}) and, as $\varepsilon \to 0$, $\nu \to 0$ according with $\nu/\epsilon \to c$, we find that it solves
\begin{align*}
\begin{array}{ll}
\displaystyle v_0^{\prime \prime} = -1 & \\
\displaystyle v_0(0) = 0 & \\
\displaystyle v_0(r_1) = 0 & \textrm{if } c=\infty\\
\displaystyle v_0^\prime(r_1) = -c \,v_0(r_1) & \textrm{if } c \in [0, \infty). 
\end{array} 
\end{align*}
This corresponds to the study of $u_n$ with $f_n = \mathbf{1}$. Due to isotropy and the discussion about the angular part of the planar BM, we arrive at the solution $u_\infty$ of the problem above. Therefore, the boundary conditions on $r_1$ depend on the limit of the ratio between   the skewness coefficient $\nu$ and the thickness coefficient $\varepsilon$. According with Section \ref{sec-Notation}, we note that $\sigma_n=1/2\pi$, $w^n=\varepsilon(n)$ and $c_n w^n=\nu(n)$, $\Sigma^n = C_{1,2}$ is the thin layer.

\paragraph*{Second approach.}
Alternatively, we can approach the problem as follows. Let $T_{c_n}$ be the stopping time for the skew BM on $C_2$ with $r_2=r_1+\varepsilon(n)$ and $\nu=\nu(n)$ under the assumption that $\lim_{n\to \infty} c_n = \lim_{n\to \infty} \nu(n)/\varepsilon(n)$. The lifetime depends on the asymptotic behaviour of the process on the collapsing annulus $C_{1,2}$.  Our result in fractal domains can be reformulated here (in regular domains) by considering the stopping time $T_{c_n}$ and the fact that $(T_{c_n} > t | B^\nu) \equiv (T_{c_n} > t | R^\nu)$  in view of the previous discussion. In particular we consider the lifetime $\zeta^{C_2}=T_{c_n}$ of $R^\nu$ and $\widehat{\zeta^{C_2}}= \tau_{C_2} \wedge \widehat{T_{c_n}}$  where $\widehat{T_{c_n}} = \inf\{s >0\, :\, L_s^{r_1} > \zeta^n \}$ with conditional law $\mathbb{P}_x(\widehat{T_{c_n}} > t | R^\nu) = \exp - c_n L^{r_1}_t$ where $L^{r_1}_t$ is the symmetric local time of $R^\nu$ at $r_1$. That is, we consider $\zeta^n$ as exponential random variable with parameter $c_n$ and independent from $R^\nu$. Thus, under the assumption that $\lim_{n\to \infty} c_n = \lim_{n\to \infty} \nu(n)/\varepsilon(n)$, we study the asymptotic behaviour of 
\begin{align*}
u_n(r) = \mathbb{E}_r\left[ \int_0^{T_{c_n}} f_n(\widetilde{R^\nu_t} )dt \right] = \mathbb{E}_r\left[ \int_0^{\infty} f_n(R^\nu_t) M^n_t dt \right]
\end{align*}
where $M^n_t = \mathbf{1}_{(t < T_{c_n})}$ by means of the asymptotic behaviour of
\begin{align*}
\widehat{u_n}(r) = \mathbb{E}_r \left[ \int_0^\infty f_n(R^\nu_t) \widehat{M^n_t} dt  \right]
\end{align*}
where $\widehat{M^n_t} = \mathbf{1}_{(t < \widehat{T_{c_n}})}$ and (assume here $x \in C_1$ for the reader's convenience)
\begin{align*}
\mathbb{P}_x(\widehat{T_{c_n}} > t | R^\nu)  \stackrel{law}{\to} \left\lbrace \begin{array}{ll}
\displaystyle \mathbf{1},  & c_n \to 0,\\
\displaystyle \exp - c_0 L^{r_1}_t, & c_n \to c_0 \in (0, \infty),\\
\displaystyle \mathbf{1}_{(t<\tau_{C_1})}, & c_n \to \infty,
\end{array} \right . \quad \textrm{as } n\to \infty.
\end{align*}
For the local times we have that $L^{r_1}_t(R^{\nu(n)}) \to L^{r_1}_t(R^+)$ in law where $R^+$ is a reflecting Bessel process on $(0, r_1)$. Thus, we estimate the stopping time $T$ by $\widehat{T}$ and exploit the fact that $\widehat{\zeta^{C_2}} \leq \zeta^{C_2}$ with probability one. This immediately follows by considering the definition of $\widehat{\zeta^{C_2}}$ which can be also written as $\widehat{\zeta^{C_2}} = \inf\{s \in(0, \zeta^{C_2}]\, :\, L_s^{r_1} > \zeta^n \}$. The convergence of $R^{\nu(n)}$ can be obtained by considering that $\mathbb{P}_r(R^{\nu(n)}_t > M) \leq M^{-1} \mathbb{E} R^{\nu(n)}_t$ and that the moment is bounded.

\begin{remark}
\label{rmk-Cap}
For a compact subset $K \in \mathbb{R}^d$ (\cite[Theorem 22.7]{RogWill1})
\begin{align*}
\mathbb{P}_x (B_t \in K \textrm{ for some } t>0) = \int G(x,y)\mu_K(dy) =G \mu_K(x)
\end{align*}
is a potential of a unique measure $\mu_K$ concentrated on $\partial K$. The capacity $\textrm{Cap}(K) = \inf \{\mathcal{E}(\mu)\,:\, G\mu \geq 1 \textrm{ on } K \}$ where $\mathcal{E}(\mu) = \int G(x,y)\mu(dx)\mu(dy)$ can be defined from $\mu_K(K)$.

Define $\sigma_K = \sup\{s >0\,;\, B_s \in K\}$ with $\sup \emptyset = 0$, then for $x\in \mathbb{R}^d$, $y \in K$, $t>0$, we have that (\cite{Chung, GetoorSharpe})
\begin{align}
\mathbb{P}_x(B_{\sigma_K} \in dy, \sigma_K \in dt) = p(t,x,y) \mu_K(dy) dt \label{hitting-desity}
\end{align}
and we recover an interesting connection between elastic coefficient and capacity. Consider $K=\overline{C_1}$: the last exit time can be therefore rewritten as $\sigma_K= \inf\{ s >0: L^{r_1}_s > \zeta^n\}$ where now $\zeta^n$ is the time the process spends on (or cross) $r_1$ before absorption in $r_2$.
\end{remark}

\begin{remark}
Notice that we used isotropy and skew product representation which are not suitable tools for approaching our fractal problem. In particular, if we consider the Koch domain $\Omega$, the normal vector does not exist at almost all boundary points.  However it is possible to define the Robin boundary condition  in the sense of the dual of certain Besov spaces (see \cite[Theorem 4.2 ]{CAP2}).
\end{remark}

\section{Transmission conditions on irregular interfaces}
\label{sec-irregular}

In this section we introduce the modified skew BM $B^{\nu,*}_t$, $t\geq 0$ on $\Omega^n_\varepsilon$. The parameter $\nu \in [0,1]$ is the so called skewness parameter.  Skew BM is a process with associated Dirichlet form in $L^2(\Omega^n_{\varepsilon}, \mathfrak{m}_\nu)$ given by
\begin{equation}
\mathcal{E}(u,v) = \frac{1}{2}\int_{\Omega^n_{\varepsilon}} \nabla u \, \nabla v \,d\mathfrak{m}_\nu, \quad D(\mathcal{E}) = H^1(\Omega^n_{\varepsilon}, \mathfrak{m}_\nu) \label{DirFormSkew}
\end{equation}
where $\mathfrak{m}_\nu(x) = 2(1-\nu) \mathbf{1}_{\Omega^n}(x) + 2\nu \mathbf{1}_{\Sigma^n}(x)$ and it can be associated with discontinuous diffusion coefficients. We focus on the sequence of elliptic operators 
\begin{equation}
L_n\, u= - \div \left( a^n_\varepsilon \,\nabla u \right) \label{div-gen}
\end{equation}
in divergence form with coefficients given in \eqref{aq} and
\begin{align*}
D(L_n)= \left\lbrace u \in L^2(\Omega^n_\varepsilon, dx),\, :\;    u|_{\Omega^n}\in H^2(\Omega^n), \quad   u|_{\Sigma^n}\in H^2(\Sigma^n) \right\rbrace.
\end{align*} 
The discontinuous coefficients $a^n_\varepsilon$ in \eqref{div-gen} introduce the transmission condition in the $L^2 (\partial \Omega^n) $
\begin{equation}
 \nabla u \cdot \mathbf{n} \big|_{y-} = c_n \sigma_n \, \nabla u \cdot \mathbf{n} \big|_{y+} \quad \forall\, y \in \partial \Omega^n \label{transm-cond}
\end{equation}
(where $ \mathbf{ n}$ is the outer normal to $ \Omega^n$, $y^-= y \in \overline{\Omega^n} \cap \partial \Omega^n$ and $y^+ = y \in \overline{\Sigma^n} \cap \partial \Omega^n$, we recall that $w^n|_{\partial \Omega^n}=1$) and therefore, the corresponding diffusion behaves like a skew BM.  For a given $n$, the operator \eqref{div-gen} can be regarded as the governing operator of the planar skew BM $\widetilde{B^\nu} = (\{\widetilde{B^\nu_t}\}_{t \geq 0}; \mathfrak{F}^\nu; \mathbb{P}^n_x, x \in \Omega^n_\varepsilon )$ on $\mathbb{R}^2$ from which we define the killed process $B^\nu$. Let $\mathcal{L}_n$ be the governing operator of $B^\nu$ on $L^2(\Omega^n_\varepsilon, dx)$ with 
\begin{align*}
\mathcal{D}(\mathcal{L}_n) =   \bigg\{& u \in L^2(\Omega^n_\varepsilon, dx), \quad   u|_{\Omega^n}\in H^2(\Omega^n), \quad   u|_{\Sigma^n}\in H^2(\Sigma^n), \\
 & :\,  u\big|_{\partial \Omega^n_\varepsilon} = 0, \quad  u \textrm{ is continuous on $\partial \Omega^n$ and satisfies \eqref{transm-cond}} \bigg\}.
\end{align*}
Then, the transition function $\mathbf{P}^n_t f(x) = \mathbb{E}^n_x[f(B^\nu_t)] = \mathbb{E}^n_x[f(\widetilde{B^\nu_t});\, t< \tau_{\Omega^n_\varepsilon}]$ with transition kernel $p^\nu$ where $\nu$ depends on the coefficients $a^n_\varepsilon$ and therefore, on \eqref{transm-cond}, is governed by
\begin{equation}
\displaystyle \frac{\partial u}{\partial t} = \mathcal{L}_n\, u  \quad \textrm{on} \quad \Omega^n_\varepsilon  \label{pdeSBM}
\end{equation}
and $\mathcal{L}_n f:= \frac{1}{2} L_n f$, $f \in \mathcal{D}(\mathcal{L}_n)$. The parabolic equation \eqref{pdeSBM} can be rewritten by considering the infinitesimal generator $\widetilde{\mathcal{L}}_n:=\frac{1}{2}\Delta$ on $L^2(\widetilde{\mathfrak{m}}_\nu)$ with $\mathcal{D}(\widetilde{\mathcal{L}}_n) = \mathcal{D}(\mathcal{L}_n)$ (see  \cite[pag 356]{chen-book} for details) where
\begin{equation}
\widetilde{\mathfrak{m}}_\nu(x) = \mathbf{1}_{\Omega^n}(x) + c_n \sigma_n w^n\, \mathbf{1}_{\Sigma^n}(x). \label{sp02}
\end{equation}
From the transition kernel $p^\nu$ we can write
\begin{equation}
\mathbb{P}^n_x(\widetilde{B^\nu_t} \in \Lambda , t < \tau_{\Omega^n_\varepsilon}) = \int_\Lambda p^\nu(t, x,y) \,  dy \quad x \in \Omega^n_\varepsilon \label{PdiSBM}
\end{equation}
for some Borel set $\Lambda \in \mathfrak{F}^\nu$ with the (first) exit time 
\begin{equation}
\tau_{\Omega^n_\varepsilon} = \inf \{s > 0\,:\, \widetilde{B^\nu_s} \notin \Omega^n_\varepsilon \} \label{exit-time-1}.
\end{equation}

We refer to $B^\nu$ as a modified skew BM in the sense that it depends on both the skewness coefficient $\nu$ (that is the BM is skew) and the weight $w^n$ given in \eqref{w(n)one} (that is, the skew BM is modified). The process $B^\nu$ represents a Brownian diffusion of a particle with transmission condition \eqref{transm-cond} on the pre-fractal $\partial \Omega^n$. 
The BM is partially reflected when it hits $\partial \Omega^n$:  that is, $\forall\, x \in \partial \Omega^n$ the process starting from $x$ moves toward $\Omega^n$ or $\Sigma^n$ with probability $1-\nu$ or $\nu$ respectively by taking into account the structural constant $\sigma_n$. In particular, according with \eqref{unif-2pi-2}, 
\begin{align}
\sigma_n \int_{\partial \Omega^n} \mathbb{P}^n_{x}( B_t^\nu  \in \Sigma^n ) d\mathfrak{s} = \nu \quad \textrm{and} \quad  \sigma_n \int_{\partial \Omega^n} \mathbb{P}^n_{x}( B_t^\nu  \in \Omega^n ) d\mathfrak{s} = 1-\nu.
\label{saythat-nu}
\end{align}
Let $\nu=\nu(n)$ be a sequence such that $\nu(n) \to 0$ as $n\to \infty$. Heuristically, \eqref{saythat-nu} and \eqref{transm-cond} say that
\begin{align*}
\frac{a^{n+}_\varepsilon}{ a^{n-}_\varepsilon + a^{n+}_\varepsilon} \frac{1}{\nu(n) \sigma_n} \to 1 \qquad \textrm{uniformly on $\partial \Omega^n$ as } n\to \infty
\end{align*}
Since condition \eqref{general-cw} holds true, from the construction we present here, it must be that $\nu(n)/c_n w^n \to 1$ on $\Sigma^n$ as $n \to \infty$. Equivalently, 
\begin{equation}
\frac{\nu(n)}{1-\nu(n)} \frac{1}{c_n w^n} \to 1 \qquad \textrm{uniformly on $\Sigma^n$ as } n\to \infty. \label{cond-nu}
\end{equation}
In view of \eqref{cond-nu}, we also refer to $\nu$ as transmission parameter. However, due to the fact that $w^n|_{\partial \Omega^n}=1$, we must pay particular attention on the pre-fractal boundary.\\

We follow the characterization of trap domain given in \cite{BBC,BCM}. Consider an open connected set $D \subset \mathbb{R}^d$, $d\geq 2$ with finite volume and the reflected BM $B^+$ on $\overline{D}$. Let $\mathcal{B} \subset D$ be an open ball with non-zero radius and denote by $\tau_{\partial \mathcal{B}} = \inf \{s\geq 0\,:\, B^+_s \in \partial \mathcal{B} \}$ the hitting time of the reflecting BM $B^+ \in D \setminus \mathcal{B}$. 
\begin{definition}
The set $D$ is a trap domain if
\begin{equation}
\sup_{x \in D \setminus \mathcal{B}} \, \mathbb{E}_x \, \tau_{\partial \mathcal{B}} = \infty. \label{trap-cond-P}
\end{equation}
Otherwise, $D$ is a non-trap domain.
\end{definition}
Notice that the definition above does not depend on the choice of $\mathcal{B}$ (\cite[Lemma 3.3]{BCM}). In both Lipschitz domains $\Omega^n$ and $\Sigma^n$ the process $B^\nu$ behaves like a BM $B^+$ reflecting on $\partial \Omega^n$. As shown in \cite{BBC, BCM}, the pre-fractal and fractal Koch domains are \emph{non-trap}. Then, $\forall\, n$, $\Omega^n$ and $\Sigma^n$ are non trap for $B^\nu$. Condition \eqref{trap-cond-P} can be rewritten in analytic way as follows
\begin{equation*}
\sup_{x \in D \setminus \mathcal{B}} \int_{D \setminus \mathcal{B}} G^+(x,y)dy =\infty
\end{equation*}
where $G^+$ is the Green function of $B^+$ on $D$ and $D=\Omega^n$ or $D=\Sigma^n$.

The process $B^\nu$ on $\Omega^n_\varepsilon$ is a transient BM for which $\mathbf{P}^n_t \mathbf{1}_{\Omega^n_\varepsilon}(x) = \mathbb{P}_x(\tau_{\Omega^n_\varepsilon} >t)$ and $\mathbb{E}_x \tau_{\Omega^n_\varepsilon} = \int \mathbf{P}^n_t \mathbf{1}_{\Omega^n_\varepsilon}(x)dt<\infty$. Nevertheless, we are looking for asymptotic results concerning also a non transient limit process. Thus, for the skew BM $B_t^\nu,\, t < \tau_{\Omega^n_\varepsilon}$, we introduce the Green function $G^\nu_n(x,y) = \int_0^\infty e^{-\delta_n t} p^\nu(t,x,y)dt$ for which we write
\begin{equation}
G^\nu_n f(x) = \int G^\nu_n(x,y) \, f(y)\, dy = \mathbb{E}^n_x \left[ \int_0^{\tau_{\Omega^n_\varepsilon}} e^{-\delta_n t}f(\widetilde{B^\nu_t}) dt \right] \label{green}
\end{equation}
where $\mathbb{E}_x^n$ is the expectation under \eqref{PdiSBM}. Furthermore, we write $G^\nu_nf(x) = \int G_n(x,y) \, f(y)\, \mathfrak{m}_\nu(y) dy$ where $G_n(\cdot, \cdot) = \int_0^\infty p(t, \cdot, \cdot)dt$ is the  Green function of a BM $B$ on $\Omega^n_\varepsilon$.

Let us consider the following representation 
\begin{equation}
\widetilde{\widetilde{B^\nu_t}} := \left\lbrace \begin{array}{lll}
B^w_t & \textrm{on }\; \mathbb{R}^2 \setminus \overline{\Omega^n} & \textrm{with probability } \nu\\
B_t & \textrm{on }\; \Omega^n & \textrm{with probability } 1- \nu
\end{array} \right .
\label{Bnu-rep}
\end{equation} 
where $B^w$ and $B$ are independent Brownian motions with $\mathbb{P}_x(B^w_t \in \mathbb{R}^2\setminus \overline{\Omega^n})=1$ and $\mathbb{P}_x(B_t \in \Omega^n)=1$ and depending on $w^n$ with $w^n \neq 1$ only outside $\overline{\Omega^n}$. Thus, $\widetilde{B^\nu_t} \stackrel{law}{=} \widetilde{\widetilde{B^\nu_t}}$, that is $\widetilde{B^\nu_t}$ equals $B^w_t$ with probability $\nu$ and $B_t$ with probability $1-\nu$; in this case, notice also that, $B$ is a reflected BM on $\Omega^n$  and $B^w$ is a BM reflected on $\partial \Omega^n$. 

We now introduce the process $\widetilde{B^{\nu, *}_t}$ which is the $m$-symmetric extension of $\widetilde{\widetilde{B^\nu_t}}$ to $(\mathbb{R}^2 \setminus \Omega^n) \cup \overline{\Omega^n}$ (see \cite[Remark 1.1]{trutnau05}, \cite[Definition 7.5.8 and Definition 7.7.1]{chen-book}). To be precise, we say that $\widetilde{B^{\nu, *}_t}$ is an $m$-symmetric extension meaning that 
\begin{equation*}
m(\partial \Omega^n) =0
\end{equation*}
where $m$ is the $2$-dimensional Lebesgue measure. The process $B^{\nu,*}$ is the part process of $\widetilde{B^{\nu,*}_t}$ on $\Omega^n_\varepsilon$ where $\widetilde{B^{\nu,*}_t}$ equals $B^{*}$ on $\overline{\Omega^n}$ (the $m$-extension of $B$ on $\overline{\Omega^n}$) and $\widetilde{B^{\nu,*}_t}$ equals $B^{w,*}$ on $\mathbb{R}^2 \setminus \Omega^n$ (the $m$-extension of $B^w$ on $\mathbb{R}^2 \setminus \Omega^n$) according with representation \eqref{Bnu-rep}.

Let us introduce the following measures on $\Omega^n_\varepsilon$
\begin{equation}
m^n_\varepsilon(dx) = \mathbf{1}_{\Omega^n \cup \Sigma^n}(x) \, dx + \mathbf{1}_{\partial \Omega^n}(x) \, d\mathfrak{s} \label{sp1}
\end{equation}
and
\begin{equation}
\mathfrak{m}^n_\varepsilon(dx) =  \mathfrak{m}_{\nu(n)}(x)\, dx + 2\nu(n) \sigma_n \mathbf{1}_{\partial \Omega^n}(x)\,d\mathfrak{s}  \label{sp2}
\end{equation}
where the measure on the pre-fractal curve is taken according with Proposition \ref{rr}. Notice that $\mathfrak{m}_\nu$ is related to $\widetilde{\mathfrak{m}}_\nu$ by means of \eqref{transm-cond}, \eqref{saythat-nu}, \eqref{cond-nu}. Thus, we write \eqref{PdiSBM} as $\mathbb{P}^n_x(B^\nu_t \in \Lambda)$ and 
\begin{equation}
\mathbb{P}^n_{{m}^n_\varepsilon}(B^{\nu, *}_t \in \Lambda) = \int_{\Omega^*} \mathbb{P}^n_x (B^{\nu, *}_t \in \Lambda)\, {m}^n_\varepsilon(dx) .
\end{equation}

\subsection{Local time and Occupation measure}
\label{subsection:localtime}
The skew BM is a Markov process with continuous paths (and discontinuous local time). The boundary local time is a PCAF defined as an occupation time process on the boundary (see \cite{chan, DPRZ} for example). Moreover, we deal with a modified skew BM depending on the weights $w^n$. For a given $n$, we introduce the occupation density $\ell^n_t(x)$, $x \in \Omega^*$, $t\geq 0$ such that, for $\Lambda \in \Omega^*$, the following occupation formula holds true
\begin{align}
\int_0^{t \wedge \tau_{\Omega^n_\varepsilon}} \mathbf{1}_{\Lambda}(\widetilde{B^{\nu,*}_s})ds = \int_0^t \mathbf{1}_{\Lambda}(B^{\nu,*}_s)ds = \int_\Lambda \ell^{n}_t(y; B^{\nu,*})\, m^n_\varepsilon(dy)= \int_\Lambda \ell^{n}_{t \wedge \tau_{\Omega^n_\varepsilon}}(y; \widetilde{B^{\nu,*}})\, m^n_\varepsilon(dy) \label{first-loc-def}
\end{align}
where $m^n_\varepsilon$ is the measure \eqref{sp1}. With some abuse of notation we do not distinguish here between absolutely continuity of the occupation density on $\Lambda \subset \Omega^n_\varepsilon$ or $\Lambda \subset \partial \Omega^n$. For the sake of simplicity we use the same symbol $\ell^n_t$ for a density w.r.t. $m^n_\varepsilon$.  In particular, with \eqref{Rev-Corr} in mind, as $t \to 0$ we have that
\begin{equation}
\begin{array}{l}
\displaystyle \frac{1}{t} \mathbb{E}^n_{m^n_\varepsilon} \left[ \int_0^t f(B^{\nu,*}_s)d\Gamma^{\Lambda}_s \right] \to \int_{\Omega^n_\varepsilon} f(x) \mathbf{1}_\Lambda (x) dx, \\
\displaystyle \frac{1}{t} \mathbb{E}^n_{m^n_\varepsilon} \left[ \int_0^t f(B^{\nu,*}_s) dL^{\Lambda}_s \right] \to \sigma_n \int_{\partial \Omega^n} f(x)\mathbf{1}_\Lambda (x) d\mathfrak{s}.
\end{array}
\label{Rev-Corr-2}
\end{equation} 
The occupation density $\ell^n_t(y; B^{\nu,*})$ must be discontinuous on $\partial \Omega^n$ (and continuous on $\Omega^n \cup \Sigma^n$ with different \lq\lq speed'' measures depending on $w^n$; recall that $w^n|_{\overline{\Omega^n}}=1$). In particular, 
$$\ell^n_t(y; B^{\nu}) dy =  \ell^n_t(y; X) \mathfrak{m}_\nu(dy)$$ 
where the process $X$ behaves like the BM \eqref{Bnu-rep} on $\Omega^n_\varepsilon$ with $\nu=1/2$ (that is, there is no reflection on $\partial \Omega^n$ for $X$). The occupation density on the boundary can be therefore written by considering the "right" (reflection from the exterior, $\Sigma^n$) and "left" ( reflection from the interior, $\Omega^n$) densities.  The symmetric local time
\begin{align}
\int_0^t \mathbf{1}_{\partial \Omega^n}(B^{\nu,*}_s) ds = L^{\partial \Omega^n}_t(B^{\nu,*}) = (L^{\partial \Omega^n-}_t + L^{\partial \Omega^n+}_t)/2 =  \sigma_n \int_{\partial \Omega^n} \ell^n_t(y; B^{\nu,*})\,  m^n_\varepsilon(dy)  \label{localT-corresp}
\end{align} 
is written in terms of $L^{\partial \Omega^n-}_t =L^{\partial \Omega^n-}_t(B^{\nu,*})$ and $L^{\partial \Omega^n+}_t=L^{\partial \Omega^n+}_t(B^{\nu,*})$, say "left" and "right" local time. In particular, $\ell^{n+}_t(y; B^{\nu,*}) = 2\nu(n) \sigma_n  \ell^n_t(y; B^{\nu,*})$ and $\ell^{n-}_t(y; B^{\nu,*}) = 2(1-\nu(n)) \sigma_n \ell^n_t(y; B^{\nu,*})$. Recall that we are dealing with the BM $B^{\nu, *}$ such that   $B^{\nu, *}=B^{*}$ on $\overline{\Omega^n}$ and $B^{\nu, *}=B^{w,*}$ on $\overline{\Sigma^n}$ with probability respectively given by $1-\nu$ and $\nu$ as in \eqref{Bnu-rep}. We have that $L^{\partial \Omega^n -}_t (B^{\nu,*})=L^{\partial \Omega^n}_t(B^*)$ and $L^{\partial \Omega^n +}_t (B^{\nu,*})=L^{\partial \Omega^n}_t(B^{w,*})$ according with \eqref{Bnu-rep} and \eqref{localT-corresp}, that is 
\begin{equation}
L^{\partial \Omega^n -}_t (B^{\nu,*}) = 2(1-\nu) L^{\partial \Omega^n}_t(B^{\nu,*}) \quad \textrm{ and } \quad  L^{\partial \Omega^n +}_t (B^{\nu,*}) = 2\nu L^{\partial \Omega^n}_t(B^{\nu,*}) \label{def-loc2}
\end{equation}
where $L^{\partial \Omega^n}_t(B^{\nu,*})$ is a symmetric local time (independent from $\nu$). Since $t \mapsto L^{\partial \Omega^n}_t(B^{\nu,*})$ is a continuous additive functional, formulas in \eqref{def-loc2} define PCAFs. Indeed, for $\eta>0$, $\eta L \in \mathbf{A}^+_c$ iff $L \in \mathbf{A}^+_c$ (\cite[Proposition VI.45.10]{RogWill2}).  The representations \eqref{def-loc2} can be obtained by considering excursions of $\widetilde{B^{\nu,*}}$ and suitable time changes for example in the case of regular interfaces as in Section 4. According with \eqref{cond-nu}, for the sequence of probabilities $\nu(n)$, it holds that 
\begin{equation}
\frac{1}{\nu(n)} \frac{c_n  w^n}{1+ c_n w^n }  \to 1 \quad \textrm{uniformly on $\Sigma^n$ as }  n \to \infty.   \label{nu-n}
\end{equation}
Observe that we always have $c_n w^n \to 0$ (as $n\to \infty$) as basic assumption between conductivity and thickness of the fiber, the insulating fractal layer case. We use the fact that, for any $f \in \mathcal{B}_+$,
\begin{align}
\int_\Lambda f(y) \ell^{n+}_t(y; B^{\nu,*}) m^n_\varepsilon(dy)  = & 2\nu(n) \sigma_n \int_\Lambda f(y) \ell^n_t(y; B^{\nu,*})\, m^n_\varepsilon(dy) \notag\\
= & \int_\Lambda f(y) \ell^n_t(y; B^{\nu,*})\, \mathfrak{m}^n_\varepsilon(dy), \quad \textrm{if} \quad \Lambda \subseteq \partial \Omega^n \label{ell-view}
\end{align}
and
\begin{align}
\int_\Lambda f(y) \ell^{n}_t(y; B^{\nu,*}) m^n_\varepsilon(dy)  = 2\nu(n) \int_\Lambda f(y) \ell^n_t(y; B^{w})\, m^n_\varepsilon(dy), \quad \textrm{if} \quad \Lambda \subseteq \Sigma^n  \label{ell-view-2} 
\end{align}
under $\mathbb{E}^n_x$. Formulas \eqref{ell-view} and \eqref{ell-view-2} can be also obtained by considering \eqref{sp2} together with representation \eqref{Bnu-rep} and by following similar arguments as in \cite{BluGet64}. Indeed, for $0 < t_1 < t_2 < \tau_{\Omega^n_\varepsilon}$, and $\Lambda = \textrm{supp}[\mu] $ where $\mu$ is the Revuz measure of $A_t$, we have that
\begin{align}
\mathbb{E}^n_x \left[ \int_{t_1}^{t_2} f(B^{\nu,*}_s) dA_s \right] = & \int_{t_1}^{t_2} ds \int_{\Omega^*} f(y) \,p^{\nu,*}(s,x,y) \, \mu(dy).  \label{ext-func-add-int}
\end{align}
If $A_t=L^{\Lambda+}_t$, then 
\begin{align}
\mathbb{E}^n_x \left[ \int_{0}^{t} f(B^{\nu,*}_s) dA_s \right] = & \int_0^t ds \int_{\Omega^*} f(y) \, p^{w,*}(s,x,y)\, \mu(y) \mathfrak{m}^n_\varepsilon(dy)\notag \\
= & 2\nu(n) \sigma_n \int_{0}^{t} ds \int_{\Lambda} f(y) \, p^{w,*}(s,x,y)\,m^n_\varepsilon(dy) \label{mean-A}
\end{align}
where $p^{w,*}(s,x,y)$ is the transition kernel of $B^{\nu,*}$ on $\Omega^n_\varepsilon \setminus \Omega^n$ and formula \eqref{ell-view} follows by \eqref{first-loc-def} and \eqref{localT-corresp}. Notice that for $\Lambda \subseteq \Sigma^n$ (that is, $A_t =\Gamma^\Lambda_t$), the integral \eqref{ell-view-2} vanishes as $n \to \infty$. 

For the Neumann heat kernel $p_N$ in an inner uniform domain, it holds that (\cite{GCoste-book}) 
\begin{equation}
c_1 t^{-1} e^{-\frac{d^2(x,y)}{c_2 t}} \leq p_N(t,x,y) \leq c_3 t^{-1} e^{-\frac{d^2(x,y)}{c_4 t}}. \label{Gbound}
\end{equation}
In view of \eqref{ext-func-add-int} and the Gaussian bound \eqref{Gbound}, there exists $C=C(t_1, t_2)>0$ such that 
\begin{equation}
\mathbb{E}^n_x \left[ \int_{t_1}^{t_2} f(B^{\nu,*}_s) dA_s \right] \leq C \int_{\Omega^*} f(y) \, \mu(dy).  \label{up-bound}
\end{equation}

It is known that the reflecting BM $B^+$ spends zero Lebesgue amount of time on the boundary $\partial \Omega$. On the other hand, we are interested in $\Gamma^{\Omega}_t$ obtained as a limit of $\Gamma^{\Omega^n_\varepsilon \setminus \Omega^n}_t$ for some $t<T$. Moreover, we focus on $L^{\partial \Omega^n+}$ and $L^{\partial \Omega^n-}$ or equivalently on $\Gamma^{\Omega^n_\varepsilon \setminus \Omega^n}_t$ and $\Gamma^{\Omega^n_\varepsilon \setminus \Sigma^n}_t$ in our analysis. In order to streamline the notation as much as possible we write $\ell^n_t$ in place of $\ell^n_t(B^{\nu, *})$ and $\ell_t$ instead of $\ell_t^\infty$ when no confusion arises.

\subsection{The probabilistic framework}

Here the aim is to provide a suitable framework to start with in the next section. We formalize some link between the previous sections and Brownian motions on trap domains, in particular on a domain with Koch interfaces. Hereafter, we assume that $d_n=0$ and $d=0$ without loss of generality. The problem in Theorem \ref{pnm} can be formulated as follows.
\begin{theorem}
The unique weak solution  of problem   \eqref{prefr.pb.varm} can be written as
\begin{align}
u_n(x) = \mathbb{E}^n_x \left[ \int_0^{\tau_{\Omega^n_\epsilon}} e^{-t \delta_n} f_n(\widetilde{B^{\nu(n), *}_t}) dt \right].  \label{sol-n}
\end{align}
\end{theorem}
The associated Dirichlet form on $H^1_0(\Omega^n_\varepsilon)$ is given by \eqref{DirFormSkew} or equivalently by \eqref{Rina}. The perturbed form $(\mathcal{E}^{\mu_{A^n}}_0, D(\mathcal{E}^{\mu_{A^n}}_0))$ is obtained by considering the Revuz measure of the additive functional $A^n_t$ associated with the killing time $\tau_{\Omega^n_\varepsilon}$. Let $\infty_{D^c}$ be the measure which is $+\infty$ on the complement $D^c$ of a Borel set $D$. Formula \eqref{semig-gen-A} becomes
\begin{equation}
\mathbf{P}^n_tf(x) = \mathbb{E}^n_x [f(B^{\nu, *}_t)] = \mathbb{E}^n_x [e^{- \overline{A^n_t}}f(\widetilde{B^{\nu,*}_t})] = \mathbb{E}^n_x [e^{-\delta_n t}f(\widetilde{B^{\nu,*}_t})\,;\, t < \tau_{\Omega^n_\varepsilon}] \label{semig-gen-A2}
\end{equation}
where $\overline{A^n_t} = \delta_n\, t   + A^n_t$ is a PCAF with drift $\delta_n$ ($\delta_n \geq 0$) and associated Revuz measure which can be written as $\mu_{\overline{A^n}}(dx) = \delta_n dx + \infty_{D^c}$ and $D=\Omega^n_\varepsilon$. The resolvent kernel is written as follows
\begin{align}
R^n_\lambda f(x) = \mathbb{E}^n_x \left[ \int_0^{\tau_{\Omega^n_\varepsilon}} e^{-\lambda t - \delta_n t} f(\widetilde{B^{\nu,*}_t}) dt \right] = \mathbb{E}^n_x \left[ \int_0^\infty e^{-\lambda t - \overline{A^n_t}} f(\widetilde{B^{\nu,*}_t}) dt\right]. \label{Res-n}
\end{align}
For the sake of simplicity we consider $\delta_n=0$ (if not otherwise specified). The case $\delta_n>0$ can be immediately obtained by considering $\overline{A^n_t}$ with $\mu_{\overline{A^n}}(dx) = \mu_{A^n}(dx) + \delta_n\, dx$ and following similar arguments. We rewrite \eqref{D-Form-General} by considering that the semigroup  \eqref{semig-gen-A2} generates the Dirichlet form on $L^2(\Omega^n_\varepsilon)$ given by
\begin{align}
\mathcal{E}^{\mu_{A^n}}_0(u,v) = \widetilde{a_n}(u,v) + \langle u,v\rangle_{\mu_{A^n}}, \quad u,v \in H^1(\mathbb{R}^2)  \cap L^2(\mu_{A^n}) \label{formEvera}
\end{align}
where
\begin{align}
\label{a-tilde}
\widetilde{a_n}(u,v) : = (1-\nu(n)) \int_{\Omega^n} \nabla u\, \nabla v\, dx +  \nu(n) \int_{\mathbb{R}^2 \setminus \Omega^n} \nabla u\, \nabla v\, dx.
\end{align}
Observe that the part process of $\widetilde{B^{\nu,*}_t}$ on $\Omega^n_\varepsilon$ is transient if and only if $\textrm{Cap}(\mathbb{R}^2 \setminus \Omega^n_\varepsilon)>0$ (\cite[Proposition 3.5.10]{chen-book}). The lifetime is finite and the process is killed. The representation \eqref{semig-gen-A2} says also that for our initial problem \eqref{prefr.pb.varm} it can be given a variational formulation as in \cite{ButtazzoDMasoMosco} by considering the measure
\begin{equation}
\label{mu-cap}
\infty_{(\Omega^n_\varepsilon)^c} (B) = \left\lbrace 
\begin{array}{ll}
\displaystyle +\infty,  & \textrm{if Cap}_1(B \cap (\mathbb{R}^2 \setminus \Omega^n_\varepsilon)) >0,\\ 
\displaystyle 0, & \textrm{otherwise}
\end{array}
\right . .
\end{equation}
Thus, the Dirichlet condition is prescribed in the capacity sense and the modified BM moves on $\mathbb{R}^2$.\\

We continue with the following representation of the solution in Theorem \ref{th.1m}.
\begin{theorem}
The unique weak solution of \eqref{eq:5m} can be written as
\begin{equation}
u(x) = \mathbb{E}_x \left[ \int_0^\infty   e^{- t \delta_0 -c_0 L_t^{\partial \Omega}} f(B^+_t)  \, dt \right] \label{sol-c0}
\end{equation}
where $B^+=(\{B^+_t\}_{t\geq 0}; \mathfrak{F}^+; \mathbb{P}_x, x \in \overline{\Omega})$ is a reflecting BM on $\overline{\Omega}$ and $L^{\partial \Omega}_t=L^{\partial \Omega}_t(B^+)$ is the local time on the boundary $\partial \Omega$.
\end{theorem}
The associated Dirichlet form, say $\mathcal{E}_0^{\mu_{A}}$, is therefore given by \eqref{Ra} with $D(\mathcal{E}^{\mu_{A}}_0) = H^1(\Omega) \cap L^2(c_0 \mu_\alpha)$. The solution \eqref{sol-c0} is obtained by considering the exponential random variable $\zeta$ with parameter $c_0>0$ (independent from $B^+$) and $\zeta^{\Omega} = \inf \{ s \geq 0\,:\, L^{\partial \Omega}_s \notin [0, \zeta] \}$. Thus, the associated semigroup is written as
\begin{align*}
\mathbf{P}^+_{t}f(x) = & \mathbb{E}_x \left[e^{-\delta_0 t}f(B_t^+); t < \zeta^{\Omega}  \right] =  \mathbb{E}_x \left[e^{-\delta_0 t} f(B_t^+); \zeta > L^{\partial \Omega}_t  \right]\\
 = & \mathbb{E}_x \left[ e^{-\delta_0 t} f(B_t^+) \, \mathbb{E}[\zeta > L^{\partial \Omega}_t \big| \mathfrak{F}^+]  \right] =  \mathbb{E}_x \left[ f(B_t^+) \, e^{- \delta_0 t -  c_0 L^{\partial \Omega}_t}  \right].
\end{align*}
Let $A_t = \Gamma^\Lambda_t(B)$ and $A^{-1}_t = \inf\{ s\geq 0\,:\, A_s \notin [0,t] \}$. Since $A_t$ is a non-decreasing process, $(A^{-1}_t < s) \equiv (A_s > t)$ and we say that $A^{-1}$ is the inverse of $A$. Obviously we have that $(\zeta^{\Omega} > t) \equiv (L^{\partial \Omega}_t < \zeta)$. It is worth mentioning that $B^+_t$ can not be written (for all $t>0$) as $B_{(A_t)^{-1}}$. We can not consider the skew product representation as in Section \ref{sec-regular} or in the recent paper \cite{take13} for instance. The reflecting BM has been investigated by many researchers and some different constructions have been also considered. Nevertheless, some technical problems can arise from the characterization of the domains. Here we consider a domain with fractal boundary and in particular, we exploit the fact that our pre-fractal and fractal Koch domains are non trap. This permits us to consider occupation measures even if the fractal nature of the boundary does not allow the study of the corresponding time changed processes. Theorem \ref{th.1Dm} can be formulated as follows.
\begin{theorem}
The unique weak solution of problem   \eqref{eq:5mD} can be written as
\begin{equation}
u(x) = \mathbb{E}_x \left[ \int_0^{\tau_\Omega} e^{-t \delta_0} f(\widetilde{B_t})dt \right] \label{sol-D}
\end{equation}
where $\tau_\Omega$ is the first time the BM $\widetilde{B}$ hits the boundary $\partial \Omega$.
\end{theorem}
The associated Dirichlet form, say $\mathcal{E}_0^{\mu_{A}}$, is given by \eqref{RaD} with $D(\mathcal{E}^{\mu_{A}}_0) = H^1(\mathbb{R}^2) \cap L^2( \infty_{\Omega^c})$.\\

We shall approach the convergence in $L^2$ of the solutions we are interested in, by first considering convergence of measures. Let $\{\mathbb{P}^n\}_n$ be a sequence of probability measures on $(E, \mathfrak{E})$. We say that $\mathbb{P}^n$ converges weakly-$\star$ to $\mathbb{P}$ on $(E, \mathfrak{E})$ as $n \to \infty$ and write $\mathbb{P}^n \stackrel{w^\star}{\to} \mathbb{P}$, if $\mathbb{E}^n f(X^n) = \int_E f d \mathbb{P}^n \to \int_E f d \mathbb{P}=\mathbb{E}f(X)$, $\forall \, f \in C_b(E)$ where $X^n$ and $X$ are the random variables with probability measures $\mathbb{P}^n$ and $\mathbb{P}$ (that is, $X^n$ convergences in law to $X$ and we also write $X^n \stackrel{law}{\to}X$). If a sequence of stochastic processes converges (weakly) in the sense of finite-dimensional laws (write $X^n \stackrel{f.d.}{\to} X$) we are in need of tightness in order to get convergence in law. Moreover, we write $X^n \stackrel{law}{\to} \infty$ (meaning also that $X^n \stackrel{a.s.}{\to} \infty$, that is almost surely or with probability one) if $\forall\, M, \; \exists n^* \, : \, \mathbb{P}(X^n>M) =1, \; \forall\, n > n^*$. We use vague convergence arguments in this case, that is for a sequence of measures $\mu_n$ on $E \cup \{+ \infty \}$ we have $\mu_n \stackrel{v}{\to} \mu$ if $\langle f, \mu_n \rangle \to \langle f, \mu \rangle$, $\forall\, f \in C_0^+$, the class of continuous functions $f:\mathbb{R} \to \mathbb{R}_+$ with compact support.

\begin{theorem} \label{Thm-M-conv-res}
(\cite{MOS1})The Mosco convergence of the forms is equivalent to the strong convergence of the associated resolvents and semigroups.
\end{theorem}

Convergence of semigroups, by the Markov property, provides convergence of finite dimensional laws. In particular (let the symbol ''$\to$'' denote strong convergence of semigroups), for the semigroup \eqref{semig-gen-A2}, under \eqref{111} and \eqref{222}, consider that (see theorems  \ref{teo1m} and \eqref{teo2}):
\begin{itemize}
\item[i)] Robin, under \eqref{333} with $c_0>0$ and $\delta_0 \geq 0$,
\begin{equation}
\mathbf{P}^n_t f_n(x) \to \mathbb{E}_x\left[ e^{-\delta_0 t}f(B^+_t)\,;\, t < \zeta^{\Omega} \right]; \label{conv-semi-R}
\end{equation}
\item[ii)]  Neumann, under \eqref{333} with $c_0=0$ and $\delta_0 >0$,
\begin{equation}
\mathbf{P}^n_t f_n(x) \to  \mathbb{E}_x\left[ e^{-\delta_0 t}f(B^+_t) \right]; \label{conv-semi-N}
\end{equation}
\item[iii)] Dirichlet, under \eqref{ISO} and $\delta_0 \geq 0$,
\begin{equation}
\mathbf{P}^n_t f_n(x) \to  \mathbb{E}_x\left[ e^{-\delta_0 t} f(\widetilde{B_t})\,;\, t < \tau_{\Omega} \right]=\mathbb{E}_x\left[ e^{-\delta_0 t} f(B_t^+)\,;\, t < \tau_{\Omega} \right]. \label{conv-semi-D}
\end{equation}
\end{itemize}
Thus, starting from the part process of $\widetilde{B^{\nu,*}_t}$ on $\Omega^n_\varepsilon$ (and therefore from \eqref{semig-gen-A2}), we simply write
\begin{equation}
\mathbf{P}^n_t f_n(x) \to \mathbf{P}_t f(x) = \mathbb{E}_x\left[ e^{-\delta_0 t} f(B^+_t);\, t < T_{c_\infty} \right] \label{conv-semi-Gen}
\end{equation}
where the stopping time depends on $\lim_{n\to \infty} c_n = c_\infty \in [0, \infty]$. We arrive at the reflecting BM on $\overline{\Omega}$ stopped by $T_{c_\infty}$, that is the lifetime depends on the asymptotic behaviour of the process on the thin layer $\overline{\Sigma^n}$.  However, the convergence in \eqref{conv-semi-Gen} follows once the convergence of a suitable sequence of stopping times to $T_{c_\infty}$ in an appropriate sense has been shown. If $m_n \to m$, for the Borel sets $\{\Lambda_j\}$ we have that
\begin{equation}
\mathbb{E}^n_{m_n} [\mathbf{1}_{\Lambda_1}(X^n_{t_1}) \cdots \mathbf{1}_{\Lambda_k}(X^n_{t_k})] \to \mathbb{E}_m [\mathbf{1}_{\Lambda_1}(X_{t_1}) \cdots \mathbf{1}_{\Lambda_k}(X_{t_k})] \label{markov-p-0}
\end{equation} 
as $n\to \infty$. This is due to the Markov property and the fact that
\begin{equation}
\mathbb{E}^n_x [\mathbf{1}_{\Lambda_1}(X^n_{t_1}) \cdots \mathbf{1}_{\Lambda_k}(X^n_{t_k})] = \mathbf{P}^n_{t_1} \mathbf{1}_{\Lambda_1} \mathbf{P}^n_{t_2 - t_1} \mathbf{1}_{\Lambda_2} \cdots \mathbf{P}^n_{t_{k} - t_{k-1}} \mathbf{1}_{\Lambda_k}(x) \label{markov-p}
\end{equation}
where $\mathbf{P}_t^n f(x)$ is the transition (non conservative) semigroup \eqref{semig-gen-A2}. Thus we have convergence of finite dimensional laws. If in addition, $\mathbb{P}_{m_n}^n$ is tight, then $\mathbb{P}_{m_n}^n$ converges weakly-$\star$ to $\mathbb{P}_m$. 
\begin{definition}
The sequence of probability measures $\{\mathbb{P}^n\}_n$ on a metric space $E$ is said to be tight if for every $\epsilon>0$, there exists a compact set $K \subseteq E$ such that $\sup_n \mathbb{P}^n(E \setminus K) \leq \epsilon$.
\end{definition}
We use the (Kolmogorov-Chentsov) criterion based on the moments of increments, that is, the sequence $X^n$ is tight if $X^n_0=0$ and there exist $\alpha, \beta >0$ and $C>0$ such that, for $T>0$,
\begin{equation}
\mathbb{E}[| X^n_t - X^n_s |^\alpha] \leq C \,|t-s|^{\beta+1} \label{mom-criterion}
\end{equation}
holds uniformly on $n \in \mathbb{N}$ and $0 \leq s,t \leq T $ (see \cite[Corollary 14.9]{Kallenberg97}). Thus, the sequence $X^n$ is tight in the space of all continuous processes, equipped with the norm of locally uniform convergence.

\section{Main results}
\label{sec-results}

We consider occupation measures on both $\Omega^n_\varepsilon$ and $\partial \Omega^n$ (local times) instead of planar Brownian motions. Let $\zeta^{\Omega^n_\varepsilon}$ be the lifetime of $B^{\nu,*}_t$ on $\Omega^n_\varepsilon$ and $\zeta^{\Omega}$ be the lifetime of the limit process on $\Omega$. Let us focus now on \eqref{conv-semi-Gen}. Let $X_t^n$ be the $m$-version of $B^{\nu(n),*}_t = \{ \widetilde{B^{\nu(n),*}_t}, \; t < T_{c_n}\}$  with transition semigroup $\mathbf{P}^n_t$ (associated with the form $\mathcal{E}^{\mu_{A^n}}_0$) and $X_t$ be the process with transition semigroup $\mathbf{P}_t$ (associated with the form $\mathcal{E}^{\mu_{A^\infty}}_0$). Our aim is to prove the following theorem.
\begin{theorem}
Let $A^n_t$ be the PCAF associated with $M^n_t=\mathbf{1}_{(t < \zeta^{\Omega^n_\varepsilon})}$ as in \eqref{semig-gen-A2}. We have:
\begin{itemize}
\item [i)] $c_n \to c_0 \in (0,\infty)$ $\Leftrightarrow$ $\zeta^{\Omega^n_\varepsilon} \stackrel{law}{\to} \zeta^{\Omega} \Leftrightarrow \mu_{A^n} \stackrel{w}{\to} \mu_{A^\infty} = c_0\, \mu_\alpha$ ($\mu_\alpha$ is defined in \eqref{mu-alpha-def}).\\ $X_t$ is an elastic (or partially reflected) BM on $\overline{\Omega}$;
\item [ii)] $c_n \to 0$ $\Leftrightarrow$ $\zeta^{\Omega^n_\varepsilon} \stackrel{a.s.}{\to} \infty \Leftrightarrow \mu_{A^n} \stackrel{w}{\to} \mu_{A^\infty} = 0$.\\ $X_t$ is a reflecting BM on $\overline{\Omega}$; 
\item [iii)] $c_n w^n \to 0$, $c_n \to \infty$ $\Leftrightarrow$ $\zeta^{\Omega^n_\varepsilon} \stackrel{law}{\to} \tau_\Omega \Leftrightarrow \mu_{A^n} \stackrel{v}{\to} \mu_{A^\infty} = \infty$ (is locally infinite).\\ $X_t$ is an absorbing BM on $\Omega$.
\end{itemize}
\label{thm-mu}
\end{theorem}
The main tools we deal with are stopping times. We first assume that a.s. $\zeta^{\Omega^n_\varepsilon} = T_{c_n}$ $\forall\, n$, that is the lifetime is equivalent to a random time depending on $c_n \geq 0$. Then, we focus on the sequence of random times $T_{c_n}$ with $c_n \to c_\infty \in [0, \infty]$ as $n \to \infty$ and we study the convergence $T_{c_n} \to T_{c_\infty}$. Thus, $T_{c_\infty}$ plays the role of lifetime for the limit BM on $\Omega$ (or $\overline{\Omega}$).

\begin{remark}
\label{rmk-life-boundary}
Let $\zeta^n$ be a r.v. with density law $ \mathbb{P}(\zeta^n \in dx) = c_n \exp(-c_n\, x) \mathbf{1}_{[0, \infty)}(x) \, dx $. We obviously have that $\mathbb{P}(\zeta^n \leq x) = 1- \exp(-c_n x)$ and $\mathbb{E}\zeta^n= 1/c_n$. We denote by $\zeta^\infty$ the limit of $\zeta^n$ as $n\to \infty$ in the following sense. 

If $c_n \to \infty$, then $\zeta^n \stackrel{\mathbb{P}}{\to} \zeta^\infty=0$. Indeed, by Markov's inequality, we have that $\mathbb{P}(|\zeta^n - \zeta^\infty| > \epsilon) \leq \epsilon^{-1} \mathbb{E}|\zeta^n |  \to 0$ as $c_n\to \infty$. Moreover, $\zeta^\infty\geq 0$ with $\mathbb{E} \zeta^\infty=0$. Thus, $\mathbb{P}(\zeta^\infty=0)=1$. 

If $c_n \to 0$ we have that $\mathbb{P}(\zeta^\infty \leq x)=0$ for all $x \in [0, \infty)$ with $\zeta^\infty \geq 0$ and, by observing that $\mathbb{E} \zeta^\infty=\infty$, we conclude that $ \mathbb{P}(\lim_{n \to \infty} \zeta^n = \zeta^\infty = \infty) = 1$. 

If $c_n\to c_0 \in (0, \infty)$, we simply have that $\zeta^n \stackrel{law}{\to} \zeta^\infty=\zeta$ as $n \to \infty$ where $\zeta$ is the exponential r.v. with parameter $c_0$.
\end{remark} 

We can also relate a r.v.  $\zeta^n$ to the time the process $B^{\nu,*}_t$ spends on (or cross) the pre-fractal $\partial \Omega^n$ as follows. For a fixed $n$, denote by $\widehat{\zeta^{\Omega^n}}$ the r.v. written as  
\begin{equation}
\widehat{\zeta^{\Omega^n}} : = \inf \{ 0 < s \leq \zeta^{\Omega^n_\varepsilon}:\, L_s^{\partial \Omega^n}(\widetilde{B^{\nu,*}}) > \zeta^n\} = \inf \{ s>0 \, :\, L_s^{\partial \Omega^n}(B^{\nu,*}) > \zeta^n\}  \label{def-life-sequence}
\end{equation}
assuming that $\zeta^n$ is independent from $B^{\nu,*}$ (and therefore, from the local time on the pre-fractal boundary). Obviously, $\zeta^{\Omega^n}$ is a sequence of Markov stopping times before absorption on $\partial \Omega^n_\varepsilon$. To be clear, $\widehat{\zeta^{\Omega^n} } \leq \zeta^{\Omega^n_\varepsilon}$ with probability one, $\forall\, n \in \mathbb{N}$ and \eqref{def-life-sequence} can be rewritten as
\begin{align*}
\widehat{\zeta^{\Omega^n}} = \tau_{\Omega^n_\varepsilon} \wedge \widehat{T_{c_n}}
\end{align*}
in terms of the first exit time $\tau_{\Omega^n_\varepsilon}$ and $\widehat{T_{c_n}} := \inf\{s>0\,:\, L^{\partial \Omega^n}_s (\widetilde{B^{\nu, *}_s})> \zeta^n\}$ with $\zeta^n$ exponentially distributed, that is $\widehat{T_{c_n}}$ corresponds to the elastic boundary condition. We also observe that \eqref{def-life-sequence} can be regarded as the lifetime of the process up to the last visit on $\overline{\Omega^n}$ assuming that $\zeta^n$ is the time the process spends on the pre-fractal boundary before absorption on $\partial \Omega^n_\varepsilon$. In this case we have the lifetime on $\Omega^n$ written as (see also Remark \ref{rmk-Cap})
\begin{equation}
\zeta^{\Omega^n} = \sup \{s> 0\,:\, B^{\nu(n),*}_s \in \overline{\Omega^n} \} =  \sigma_{\overline{\Omega^n}}  \label{def-life-sequence-2}
\end{equation}
which is no longer Markovian. Moreover, $\mathbb{P}_x(\widehat{\zeta^{\Omega^n_\varepsilon}} > \zeta^{\Omega^n}) > 0$, $m$-a.e. $x$ and 
\begin{align*}
\mathbb{P}_x(\sigma_{\overline{\Omega^n}} > 0) = \mathbb{P}_x(\tau_{\Omega^n} < \infty). 
\end{align*}

We consider \eqref{def-life-sequence} with exponential threshold $\zeta^n$ given in Remark \ref{rmk-life-boundary}.

\begin{remark}
\label{rmk-life-domain}
From the discussion in the previous remark, we have the following cases.\\
If $c_n \to \infty$, then $\zeta^n \stackrel{\mathbb{P}}{\to} 0$ and
\begin{align}
\widehat{\zeta^{\Omega^n}} \stackrel{law}{\to} \inf\{s > 0\,:\, L^{\partial \Omega}_s(B^+) >0 \} =: \tau_\Omega \label{maps-one}
\end{align}
provided that the occupation time sequence $L^{\partial \Omega^n}_t$ converges in law to the local time $L^{\partial \Omega}_t$ of the corresponding limit process.\\ 
Similarly, as $c_n \to 0$, $\mathbb{P}(\lim_{n \to \infty} \zeta^n  = \infty) = 1$ and then 
\begin{align}
\widehat{\zeta^{\Omega^n}} \stackrel{law}{\to} \inf\{s > 0\,:\, L^{\partial \Omega}_s(B^+) > \infty \} = \infty. \label{maps-two}
\end{align}
If $c_n \to c_0 \in (0,\infty)$, then $\zeta^n \stackrel{law}{\to} \zeta$ and the lifetime of the limit process depends on the random variable $\zeta$ with parameter $c_0$. In particular, we have that
\begin{equation}
\widehat{\zeta^{\Omega^n}} \stackrel{law}{\to} \zeta^{\Omega} = \inf\{ s> 0\,:\, L^{\partial \Omega}_s(B^+) > \zeta \} \label{maps-three}
\end{equation}
provided the convergence in law of the occupation time process on the fractal boundary $\partial \Omega$.
\end{remark}

For the sake of simplicity we write $\zeta^{\Omega}$ in place of $\zeta^{\Omega^\infty}$. We focus on the PCAF
\begin{equation}
\widetilde{A_t^{n+}} =\Gamma^{\Sigma^n}_t(\widetilde{B^{\nu,*}}) + L^{\partial \Omega^n+}_t(\widetilde{B^{\nu,*}}) = \Gamma^{\Omega^n_\varepsilon \setminus \Omega^n}_t(\widetilde{B^{\nu,*}}) \label{ThmA}
\end{equation}
and $\int_0^{t \wedge \tau_{\Omega^n_\varepsilon}} d\widetilde{A^{n+}_s}$ is the occupation time process on $\Omega^n_\varepsilon \setminus \Omega^n$ of the skew BM $B^{\nu,*}$. Let us write
\begin{align*}
R^n_\lambda \mu_n^+(x) = & 2 (1+c_n)^{-1}c_n \, \sigma_n\, \mathbb{E}^n_x \left[ \int_0^{\zeta^{\Omega^n_\varepsilon}} e^{-\lambda t - t \delta_n}\mathbf{1}_{\partial \Omega^n}(\widetilde{B^{\nu(n),*}_t})dt \right]
\end{align*}
where $\mu_n^+(dx) = 2 (1+c_n)^{-1} c_n \sigma_n \, \mathbf{1}_{\partial \Omega^n}(x)\, m^n_\varepsilon(dx)$. Let $A^{n+}_t \in \mathbf{A}^+_c$ be in Revuz correspondence with the measure $\mu_n^+$.

\begin{proposition}
\label{prop-last}
Under \eqref{555}, the boundary local time $\{L^{\partial \Omega+}_t\,,\, t < \zeta^{\partial\Omega} \}$ is the unique PCAF such that, for any $x \in \Omega^n_\varepsilon$, 
\begin{equation}
R^n_\lambda \mu_n^+(x) \to  \mathbb{E}_x \left[ \int_0^{\zeta^{\Omega}} e^{-\lambda t - \delta_0 t} \,dL^{\partial \Omega+}_t \right] \quad  \textrm{as} \quad n\to \infty. \label{prop-last-eq}
\end{equation}
\end{proposition}
\begin{proof}
First we notice that $\mu_n$ has finite energy integral. Indeed, 
\begin{align*}
\int_{\Omega^*} v(x)\,\mu_n^+(dx) \leq \left( \mu_n^+(\partial \Omega^n) \right)^{1/2}\, \| v \|_{L^2(\partial \Omega^n)} = 2 (1+c_n)^{-1} c_n \sqrt{\sigma_n} \| v \|_{L^2(\partial \Omega^n)}.
\end{align*}
From \cite[Theorem 8.1]{CV-asy} we know that
\begin{equation*}
\| v \|^2_{L^2(\partial \Omega^n)} \leq \frac{C}{\sigma_n} \| v \|^2_{H^1(\mathbb{R}^2)}
\end{equation*}
where $C$ is independent of $n$. Since $(1+c_n)^{-1} c_n \leq 1$, by extension theorem (see \cite[Theorem A.3]{CV-asy}) we obtain that $\langle v, \mu_n^+ \rangle \leq \sqrt{C}\, \|v\|_{H^1(\Omega^n_\varepsilon)}$ and $\mu_n^+ \in S_0$. Since, under \eqref{555},  $R^n_\lambda \mu_n^+(x)$ is bounded (and in view of \cite[Lemma 4.1.5]{chen-book}) we have that $\mu_n^+ \in S_{00}$. Let $c_\infty\in [0, \infty]$ be such that $c_n \to c_\infty$. From \eqref{cnH}, for all $f \in \mathcal{B}_+$ we get that 
\begin{align*}
\langle f, \mu_n^+ \rangle  \to 2(1+c_\infty)^{-1} c_\infty \int_{\partial \Omega} f\,d\mu \quad \textrm{as} \quad n \to \infty.
\end{align*}
For a fixed $n$, consider the occupation measure \eqref{ThmA}. In view of \eqref{first-loc-def}, \eqref{localT-corresp}, \eqref{ell-view} and \eqref{ell-view-2}, we write
\begin{align*}
\mathbb{E}^n_x \left[ \int_0^{t \wedge \tau_{\Omega^n_\varepsilon}} f(\widetilde{B^{\nu(n),*}_s}) d\widetilde{A^{n+}_s} \right] = & 2\nu(n) \, \mathbb{E}^n_x \left[  \int_0^{t \wedge \tau_{\Omega^n_\varepsilon}} f(\widetilde{B^{\nu,*}_s}) d\Gamma^{\Sigma^n}_s + \int_0^{t \wedge \tau_{\Omega^n_\varepsilon}} f(\widetilde{B^{\nu,*}_s}) dL^{\partial \Omega^n}_s  \right]
\end{align*}
where $\Gamma^{\Sigma^n}_t = \Gamma^{\Sigma^n}_t(B^w)$. Set $U_n^\lambda f(x) = \mathbb{E}^n_x \left[ \int_0^\infty e^{-\lambda t - \delta_n t} f(\widetilde{B^{\nu, *}_t}) \, \mathbf{1}_{(t < \zeta^{\Omega^n_\varepsilon})} d\widetilde{A^{n+}_t}  \right]$. From \eqref{Rev-Corr-2} and the fact that $\mathbb{E}^n_x[\mathbf{1}_{(0 < \zeta^{\Omega^n_\varepsilon})}]= \mathbb{P}^n_x ( \zeta^{\Omega^n_\varepsilon} > 0) = 1$ for all $x \in \Omega^n_\varepsilon$ we obtain 
\begin{align}
\lim_{\lambda \to \infty} \langle \lambda U_n^\lambda f, m^n_\varepsilon \rangle = & \langle f, \widetilde{\mu_n}^+ \rangle, \quad \forall\, f \in \mathcal{B}_+ \label{lim-prop-R}
\end{align}
where $\widetilde{\mu_n}^+(dx) = 2\nu(n)(\mathbf{1}_{\Sigma^n} + \sigma_n \mathbf{1}_{\partial \Omega^n}) \,m^n_\varepsilon(dx)$ on $\Omega^n_\varepsilon \setminus \Omega^n$ 
and therefore, by \cite[Theorem 5.1.4]{FUK-book} there exists a unique PCAF of $B^{\nu,*}_t$ in Revuz correspondence with $\widetilde{\mu_n}^+$, that is  \eqref{ThmA}.  Notice also that, since $\nu(n) \leq 1$, we can follow the same arguments as before in order to see that $\widetilde{\mu_n}^+ \in S_{0}$. Let us consider $\xi_n \in \Sigma^n$ and $0 \leq M_n < \infty$. Since condition \eqref{nu-n} holds true, we get that
\begin{align*}
\lim_{n \to \infty} \langle f, \widetilde{\mu_n}^+ \rangle = \lim_{n \to \infty} \left( 2 M_n \frac{c_n w^n(\xi_n)}{1+ c_n w^n(\xi_n)} f(\xi_n) + 2\frac{c_n}{1+c_n} \sigma_n \int_{\partial \Omega^n} f d \mathfrak{s} \right)
\end{align*}
and therefore, from Proposition \ref{rr} and \eqref{general-cw},
\begin{align*}
\langle f, \widetilde{\mu_n}^+ \rangle \to 2(1+c_\infty)^{-1} c_\infty \int_{\partial \Omega} fd\mu, \quad as\ n\to \infty
\end{align*}
where $\mu=\mu_\alpha$ is the only Borel measure which is in Revuz correspondence with the ''symmetric'' local time $L^{\partial \Omega}_t$. Indeed, according with \eqref{Rev-Corr} and the setting in \eqref{D-Form-General} with $E=\overline{\Omega}$, we have that
\begin{align*}
\lim_{\lambda \to \infty} \lambda \mathbb{E}_{m } \left[ \int_0^\infty e^{-\lambda t} f(B_t) dL^{\partial \Omega}_t \right] = & \int_{\partial \Omega} fd\mu. 
\end{align*}
Notice that
\begin{equation*}
R^n_\lambda \mu_n^+(x) \to 2 (1+c_\infty)^{-1}c_\infty \, \mathbb{E}_x \left[ \int_0^{\zeta^{\Omega}} e^{-\lambda t - \delta_0 t}\mathbf{1}_{\partial \Omega}(B^+_t)dt \right] < \infty
\end{equation*}
according with \eqref{conv-semi-Gen}. From the one to one correspondence between \eqref{ThmA} and its Revuz measure $\widetilde{\mu_n}^+$ we prove \eqref{prop-last-eq} and the claim.
\end{proof}

We also focus on the PCAF
\begin{align}
\label{ThmAmeno}
\widetilde{A^{n-}_t} = \Gamma^{\Omega^n \setminus D^n}_t(\widetilde{B^{\nu,*}}) + L^{\partial \Omega^n-}_t (\widetilde{B^{\nu,*}}) = \Gamma^{\Lambda_n}_t(\widetilde{B^{\nu,*}})   
\end{align}
where $D^n_j$ is an increasing  sequence of open sets such that  $D^n=\cup_j D^n_j \to \Omega$ and $\Lambda_n = \overline{\Omega^n} \setminus D_n $. Let us write
\begin{align*}
R^n_\lambda \mu_n^-(x) = & 2 (1+c_n)^{-1} \, \sigma_n\, \mathbb{E}^n_x \left[ \int_0^{\zeta^{\Omega^n_\varepsilon}} e^{-\lambda t - t \delta_n}\mathbf{1}_{\partial \Omega^n}(\widetilde{B^{\nu(n),*}_t})dt \right]
\end{align*}
where $\mu_n^-(dx) = 2 (1+c_n)^{-1} \sigma_n \, \mathbf{1}_{\partial \Omega^n}(x)\, m^n_\varepsilon(dx)$. Let $A^{n-}_t \in \mathbf{A}^+_c$ be in Revuz correspondence with the measure $\mu_n^-$.

\begin{proposition}
\label{prop-lastt}
Under \eqref{555}, the boundary local time $\{L^{\partial \Omega-}_t, \, t < \zeta^{\Omega}\}$ is the unique PCAF such that, for any $x \in \Omega^n_\varepsilon$,
\begin{equation}
R^n_\lambda \mu_n^-(x) \to \mathbb{E}_x \left[ \int_0^{\zeta^\Omega} e^{-\lambda t - \delta_0 t} dL^{\partial \Omega -}_t \right], \quad \textrm{as} \; n \to \infty.
\end{equation}
\end{proposition}
\begin{proof}
We basically follows the proof of Proposition \ref{prop-last}. Indeed, we can write the left local time by considering \eqref{ThmAmeno} instead of \eqref{ThmA} and the density $\ell^{n-}_t(y) = 2(1-\nu(n)) \sigma_n \ell^n_t(y)$ as indicated in Section \ref{subsection:localtime}. We get that $\widetilde{\mu_n}^-(dx) = 2(1-\nu(n))\sigma_n m^n_\varepsilon(dx)$ is the Revuz measure of $\widetilde{A^{n-}_t} $. The result follows from the previous proof of Proposition \ref{prop-last}.
\end{proof}

From the previous results we have that $\mu_n^+$ is the Revuz measure of 
\begin{align*}
A^{n+}_t = 2\frac{c_n}{1+c_n} \sigma_n \int_{\partial \Omega^n} \ell^n_t(y)m^n_\varepsilon(dy) = 2 \frac{c_n}{1+c_n} L^{\partial \Omega^n}_t
\end{align*}
and $\mu_n^-$ is the Revuz measure of
\begin{align*}
A^{n-}_t = 2\frac{1}{1+c_n} \sigma_n \int_{\partial \Omega^n} \ell^n_t(y)m^n_\varepsilon(dy) = 2 \frac{1}{1+c_n} L^{\partial \Omega^n}_t.
\end{align*}
As we can immediately see $(A^{n+}_t + A^{n-}_t)/2 = L^{\partial \Omega^n}_t$. Nevertheless, $A^{n+}_t$ and  $A^{n-}_t$ behave respectively like $L^{\partial \Omega^n +}_t$ and $L^{\partial \Omega^n -}_t$ only as $n\to \infty$.
As we observed in Section \ref{subsection:localtime}, the local time $L^{\partial \Omega^n}_t$ of the  process $B^{\nu,*}_t$ is given by  
\begin{align}
L^{\partial \Omega^n}_t = \sigma_n \int_{\partial \Omega^n} \ell^n_t(y)\, m^n_\varepsilon(dy) . \label{localT-corresp-ext}
\end{align}
According with Proposition \ref{rr}, we now study the convergence of \eqref{localT-corresp-ext}  to 
\begin{align}
\label{localT-fract}
L^{\partial \Omega}_t= \int_{\partial \Omega} \ell_t(y)\, d\mu  
\end{align}
with $L^{\partial \Omega}_t \in \mathbf{A}^+_c$ (see for example \cite{BBC} for the existence and other properties of \eqref{localT-fract}). The connection between tightness on the line and continuity of the limit process has been pointed out starting from \cite{AldousI, AldousII}.
The convergence in law of \eqref{localT-corresp-ext} is proved in the following Proposition \ref{Prop-tight} and Proposition \ref{prop-L-w}. We begin with the following result concerning the tightness of the sequence $L^{\partial \Omega^n}_t$, $t < T$.

\begin{proposition} \label{Prop-tight}
The sequence $\{ L^{\partial \Omega^n}_t\}_{n}$ is tight in $C([0,T], [0, \infty))$.
\end{proposition}
\begin{proof}
Let us consider the sets $\Lambda_n \subseteq \partial \Omega^n$. The process $L^{\Lambda_n}_t,\, t < T$ is a continuous additive functional of zero energy (\cite[pag. 149]{chen-book}) and $\mathbb{E}_x[|L^{\Lambda_n}_t|,\, t < T]< \infty$ q.e. $x$. Indeed, from \eqref{def-loc2} and \eqref{up-bound}, we have that
\begin{align}
\mathbb{E}^n_x\left[ \int_s^t dL^{\Lambda_n-}_u \right] + \mathbb{E}^n_x\left[ \int_s^t dL^{\Lambda_n+}_u \right] = & 2 \mathbb{E}^n_x\left[ \int_s^t dL^{\Lambda_n}_u \right]  \notag \\
\leq & \, const \cdot (t-s) \cdot  \sigma_n \int_{\Lambda_n} m^{n}_\varepsilon(dy) . \label{bound-tight-RN}
\end{align}
We have that $L^{\Lambda_n}_0=0$ for all $n$ and, for $k>1$ and $c_1 >0$,
\begin{equation}
\mathbb{E}_x [ | L^{\Lambda_n}_t - L^{\Lambda_n}_s |^k ] \leq \sigma_n\, m^n_\varepsilon(\Lambda_n)\, c_1\, |t-s|^k < c_1\, |t-s|^k. \label{mom-crit-1}
\end{equation}
Indeed, for $s<t$,
\begin{align}
\mathbb{E}_x \left[ \int_s^t dL^{\Lambda_n}_z  \right]^k = & k! \int_s^{z_1} \ldots \int_{z_{k-1}}^t \mathbb{E}_x\left[\mathbf{1}_{\Lambda_n}(B^{\nu(n), *}_{z_1}) \ldots \mathbf{1}_{\Lambda_n}(B^{\nu(n), *}_{z_k}) \right] dz_1 \ldots dz_k \label{momentK-localT-1}\\
\leq & k! \int_s^t \ldots \int_s^t \mathbb{E}_x\left[\mathbf{1}_{\Lambda_n}(B^{\nu(n), *}_{z_1}) \ldots \mathbf{1}_{\Lambda_n}(B^{\nu(n), *}_{z_k}) \right] dz_1 \ldots dz_k. \label{momentK-localT}
\end{align}
We recall \eqref{Gbound} and the fact that $\mathbf{P}_t^D \leq  \mathbf{P}_t^N$ a.e., the transition function of the reflecting BM dominates that of the absorbing BM. From \eqref{markov-p} and \eqref{bound-tight-RN} we can write  \eqref{mom-crit-1}. Since $k>1$, a Kolmogorov-type criterion shows the tightness.
\end{proof}

Since $\widehat{\zeta^{\Omega^n}} \leq \zeta^{\Omega^n_\varepsilon}$ we can write $\mathbb{P}_x(\zeta^{\Omega^n_\varepsilon} \leq t) \leq  \mathbb{P}_x(\widehat{\zeta^{\Omega^n}} \leq t) =  \mathbb{P}_x(\zeta^n \leq L_t^{\partial \Omega^n}) = 1 - \mathbb{E}_x e^{-c_n L_t^{\partial \Omega^n}}$ and, equivalently $\mathbb{P}_x(\zeta^{\Omega^n_\varepsilon} > t) \geq  \mathbb{E}_x e^{-c_n L_t^{\partial \Omega^n}}$. Assume that $u$ is the bounded solution to $\partial_t u = \mathcal{L}_n\, u - \widehat{\kappa_n}(x) u$ with $u_0=1$ in $H^1(\mathbb{R}^2) \cap L^2(\mathbb{R}^2, \sigma_n \mathbf{1}_{\partial \Omega^n} d\mathfrak{s})$ where $\widehat{\kappa_n} = c_n \mathbf{1}_{\partial \Omega^n}$. Then,  as $t\to 0$, $t^{-1}(1 - u(x,t)) \to  - \mathcal{L}_n\, u_0 + \widehat{\kappa_n}(x) u_0 =  \widehat{\kappa_n}(x)$ gives the killing rate which can be also obtained as (\cite{BluGet68}) 
$$\lim_{t \to 0} t^{-1}\mathbb{P}_x (B^{\nu, *}_t \textrm{ is killled in the time interval } (0,t]).$$ 
Since for $A_t=\int_0^t f(X_s)ds$, $de^{A_t}/dt = f(X_t)e^{A_t}$ and  $e^{A_t} - 1 = \int_0^t f(X_s) e^{A_s} ds$, we get that
\begin{align*}
\lim_{t\to 0} t^{-1} (1 - \mathbb{E}_x e^{-c_n L_t^{\partial \Omega^n}}) = \lim_{t \to 0} \mathbb{E}_x \left[ c_n  \mathbf{1}_{\partial \Omega^n}(B^{\nu,*}_t) \right] =  \widehat{\kappa_n}(x).
\end{align*}
Then, $u(x,t)=\mathbb{P}_x^n(\widehat{\zeta^{\Omega^n}}>t)$. We use the symbol $\widehat{\kappa_n}$ in order to underline the connection with the transition semigroup
\begin{align}
\widehat{\mathbf{P}^n_t}f(x) := \mathbb{E}^n_x[f(\widetilde{B^{\nu(n),*}_t}); t < \widehat{\zeta^{\Omega^n}}] ) =  \mathbb{E}^n_x[e^{-\widehat{A^n_t}} f(B^{\nu(n),*}_t) ] \label{semig-estim}
\end{align} 
with resolvent $\widehat{R^n_\lambda}$ where $\widehat{\zeta^{\Omega^n}}$ is defined as in \eqref{def-life-sequence}. For $\delta_n$ identically zero, 
$$\widehat{\mathbf{P}^n_t} \mathbf{1}_{\Omega^n_\varepsilon}(x) = 
\mathbb{E}_x[\mathbf{1}_{\Omega^n_\varepsilon}(\widetilde{B^{\nu(n),*}_t}); t < \widehat{\zeta^{\Omega^n}}] = \mathbb{P}^n_x (\widehat{\zeta^{\Omega^n}} > t).$$ 
If $\delta_n >0$ we have that $\widehat{\mathbf{P}^n_t} \mathbf{1}_{\Omega^n_\varepsilon}(x) = e^{-\delta_n t} \, \mathbb{P}^n_x (\widehat{\zeta^{\Omega^n}}>t)$.

The heat equation solution with Robin boundary conditions has been studied using a Feynamn­-Kac formula and a theorem of Ray and Knight on Brownian local time in \cite{BBC}. In \cite{FitzPit99}, the authors reviewed Kac's method by underlining the connection between the higher-order moments  of $A_\kappa= \int_0^T \kappa(X_t)dt$ and the Feynman-Kac formula. Thus, $\mathbb{E}_m[(A_\kappa)^k] = k! m G_\kappa^k \mathbf{1}$ where $G_\kappa w(x)=\int G(x,dy)\kappa(y)w(y)$ (here $m$ is an arbitrary initial distribution). In particular,   
\begin{align}
f(x)=\mathbb{E}_x \left[ \exp \int_0^T \kappa(B_t)dt \right]=\sum_{k=0}^\infty G^k_\kappa \mathbf{1}(x), \quad x \in \Lambda \label{hig-mom-FK-formula}
\end{align}
is finite if and only if is the minimal solution to $f(x) =  1 + \int \kappa(y)f(y)G(x,dy)$ where $\kappa: \Lambda \mapsto [0,\infty)$ is a measurable Borel function, $T=\tau_\Lambda$ and $G$ is the Green function of the killed BM on the boundary of the open set $\Lambda$. 

\begin{proposition}
\label{prop-L-w}
For $x \in \Omega^1_\varepsilon$, $ \forall\, t\geq 0$, 
\begin{align*}
\int_0^\infty e^{-c_n l} \, \mathbb{P}^n_x(L^{\partial \Omega^n}_t \in dl ) \to \int_0^\infty e^{-c_\infty l}\, \mathbb{P}_x( L^{\partial \Omega}_t \in dl), \quad \textrm{as } n \to \infty
\end{align*}
$m$-a.e. $x$, for every sequence $c_n \geq 0$ such that $c_n \to c_\infty \in [0, \infty]$.
\end{proposition}
\begin{proof}
The higher-order moment of the boundary local time can be written as in formula \eqref{momentK-localT-1} with $s=0$. Then, by applying \eqref{markov-p-0} we have that $\mathbb{E}^n_x \left[( L^{\partial \Omega^n}_t )^k\right] \to \mathbb{E}_x \left[(L^{\partial \Omega}_t )^k\right]$ weakly (continuously on $t\geq 0$, $m$-a.e. $x$). Indeed, we have weak convergence of finite-dimensional distributions from Theorem \ref{Thm-M-conv-res} and  tightness from Proposition \ref{Prop-tight}. Consider the expansion \eqref{hig-mom-FK-formula}. We can also write $v_{\widehat{\kappa_n}}(x,t) = \mathbb{E}^n_x \exp - c_n L^{\partial \Omega^n}_{t}$ for $t \leq T_{c_n}$ in terms of \eqref{momentK-localT-1} with $s=0$. For $c_n=\lambda\geq 0$ $\forall\, n$, $v_{\widehat{\kappa_n}}$ is the Laplace transform of $\mathbb{P}^n_x(L^{\partial \Omega^n}_t \in dl)/dl$. Since, $\mathbb{P}_x^n \stackrel{w^\star}{\to} \mathbb{P}_x$, $m$-a.e. $x$ (from the convergence of moments), then the convergence holds for every $c_n \geq 0$ (and therefore, for every $c_\infty \in [0, \infty]$). 
\end{proof}

We can write $A^n_t = \int_0^t \kappa_n(\widetilde{B^{\nu(n),*}_s})ds$ in \eqref{semig-gen-A2} where $\kappa_n(\cdot)$ is the state dependent rate for the Markov killing time $T_{c_n}$. Thus, we can focus on the multiplicative functional $M^n_t = e^{-A^n_t}$ associated with the stopping time $T_{c_n}$ and therefore, with the sub-Markov semigroup \eqref{conv-semi-Gen}. In particular, \eqref{conv-semi-Gen} characterizes $M^n_t$ uniquely (\cite[Proposition 1.9]{BluGet68}). Also we write $\overline{M^n_t} = e^{-\overline{A^n_t}}$ (see formula \eqref{semig-gen-A2}). Notice that (for $\lambda \geq 0$)
\begin{align*}
R^n_\lambda f(x) = \mathbb{E}^n_x \left[ \int_0^\infty e^{-\lambda t} f(\widetilde{B^{\nu,*}_t})  \overline{M^n_t} dt \right] \quad \textrm{and} \quad \widehat{R^n_\lambda} f(x) = \mathbb{E}^n_x \left[ \int_0^{\tau_{\Omega^n_\varepsilon}} e^{-\lambda t-\delta_n t} f(\widetilde{B^{\nu,*}_t})  \widehat{M^n_t} dt \right]
\end{align*}
where $\widehat{M^n_t} = \exp (- \widehat{A^n_t})$ with $\delta_n, c_n$ as in \eqref{555} and $\widehat{A^n_t} = \int_0^t \widehat{\kappa_n}(\widetilde{B^{\nu(n),*}_s})ds =c_n L^{\partial \Omega^n}_t$. 

The fundamental principle in our investigation is to approximate $M$ by $\widehat{M}$ and transfer the properties known for the approximating functional.

\begin{theorem}
\label{prop-exp}
$\widehat{\zeta^{\Omega^n}} \stackrel{law}{\to} T_{c_\infty}$ and under \eqref{111}, $\widehat{R^n_\lambda} f_n \to R_\lambda f$ strongly in $L^2(\Omega)$.
\end{theorem}
\begin{proof}
First we show that $\widehat{\zeta^{\Omega^n}} \stackrel{law}{\to} T_{c_\infty}$ as $c_n \to c_\infty$. Let us write $\zeta^n= \varsigma / c_n$ with $\mathbb{P}(\varsigma > x) = e^{-x}$, $x \geq 0$. Thus,  a.s. $\zeta^n \to \zeta^\infty$. Assume $\delta_n=0$ and $\lambda>0$. From $\widehat{\mathbf{P}^n_t}$ we write the associated resolvent 
\begin{align*}
\widehat{R^n_\lambda} \mathbf{1}_{\Omega^n_\varepsilon}(x) = \int_0^\infty e^ {-\lambda t} \widehat{\mathbf{P}^n_t} \mathbf{1}_{\Omega^n_\varepsilon}(x) dt = \int_0^\infty e^ {-\lambda t}\mathbb{P}^n_x(\widehat{\zeta^{\Omega^n}} > t)  dt
\end{align*}
and, from the resolvent \eqref{Res-n},
\begin{align*}
R_\lambda^n \mathbf{1}_{\Omega^n_\varepsilon}(x) = \int_0^\infty e^{-\lambda t} \mathbf{P}^n_t \mathbf{1}_{\Omega^n_\varepsilon}(x) dt= \int_0^\infty e^{-\lambda t} \mathbb{P}^n_x(\zeta^{\Omega^n_\varepsilon} >t) dt.
\end{align*}
Recall that \eqref{conv-semi-Gen} holds true. That is, from  M-convergence of forms (\cite{CV-asy}), we have that (Theorem \ref{Thm-M-conv-res})
\begin{align*}
R_\lambda^n \mathbf{1}_{\Omega^n_\varepsilon}(x)  \to \int_0^\infty e^{-\lambda t} \mathbb{P}_x(T_{c_\infty} > t) dt = R_\lambda \mathbf{1}_{\Omega}(x)
\end{align*}
strongly in $L^2(\Omega)$. The limit stopping time $T_{c_\infty}$ depends on $c_\infty$ and corresponds to the lifetime of the limit process. By construction we have that a.s. $\widehat{\zeta^{\Omega^\infty}} = T_{c_\infty}$ and
\begin{align}
\lim_{n\to \infty} R^n_\lambda \mathbf{1}_{\Omega^n_\varepsilon}(x) = \lim_{n\to \infty} \widehat{R^n_\lambda} \mathbf{1}_{\Omega^n_\varepsilon}(x), \quad \forall\, x \in \Omega^1_\varepsilon . \label{conv-a.e.R-proof}
\end{align}
Now we show convergence in law.

\emph{Step 1)} Let us consider a reflecting BM on $\Omega^n$, say $X^n$, and the first hitting time $\tau_{\Omega^n}$ on the pre-fractal boundary $\partial \Omega^n$. Let $X$ be a reflecting BM on $\Omega$ with $\Omega^n \uparrow \Omega$. We have that $X^n \stackrel{law}{\to} X$ as $n\to \infty$ (\cite[Theorem 4.4 and Remark 2]{ChenPTRF93}, \cite[Section 2]{BurdzyChenECP98}). Consider now the killed process $X^n_t, t < \tau_{\Omega^n}$ started at $x \in \Omega^1_\varepsilon$. Here we can follow the same arguments as in \cite[Theorem 4.1 and Theorem 3.3 ]{BaxterDalMasoMosco}. In particular, we consider the additive functional $A^n_t$ associated with $\tau_{\Omega^n}$ with Revuz measure $\mu_{A^n} = \infty_{D_n}$ where $D_n$ is the complement of $\Omega^n$. From stable convergence of multiplicative functionals we arrive at weak convergence of stopping times, that is $\tau_{\Omega^n} \stackrel{law}{\to} \tau_\Omega$. Therefore, we get a direct consequence of M-convergence of the associated form and tightness of $\mathbb{P}^n_x(\tau_{\Omega^n}>t)$ (indeed $\mathbb{E}_x^n \tau_{\Omega^n}$ is bounded by $\mathbb{E}^n_x \tau_{\Omega^1_\varepsilon}$ uniformly on $n$). Notice also that we arrive at the same result by considering Dirichlet condition on $\partial \Omega^n $ and the corresponding (killed) process on $\Omega^n$. The variational approach has been treated in \cite{CAP2}.
 
\emph{Step 2)} Now we focus on the family $\{\mathbb{P}^n_x\,;\, x \in \Omega^n_\varepsilon\}$. Let us write $X^{x,n}_t = \{ B^{\nu(n),*}_t$ started from $x \in \Omega^n_\varepsilon = \overline{\Omega^n} \cup \Sigma^n \}$ $m$-a.e.. We first observe that
\begin{align*}
\forall\, n, \quad X^{x,n}_t \in \partial \quad \textrm{ if and only if } \quad   (t > \zeta^{\Omega^n_\varepsilon}) \vee (X^{x,n}_0 \in \partial) .
\end{align*}
Since $x \in \overline{\Omega^1_\varepsilon} \setminus \Omega^n_\varepsilon \Rightarrow X^{x,n}_0 \in \partial$, for $n\to \infty$ we can consider starting points $x \in \Omega^1_\varepsilon$. Notice also that, in the Neumann case, the cemetery point is assumed to be $\partial = \{\emptyset\}$ and such that $m(\partial) = 0$. This corresponds to Cap$_1(\overline{\Omega^1_\varepsilon} \setminus \Omega^n) \to 0$. Thus, for $x \in \Omega^1_\varepsilon$ we write
\begin{align*}
\mathbb{P}^n_x(\widehat{\zeta^{\Omega^n}} >t ) = & \mathbb{P}^n_x(\widehat{\zeta^{\Omega^n}} >t , (t < \tau_{\Omega^n}) \cup (t \geq \tau_{\Omega^n}))\\
= & \mathbb{E}^n_x\left[ e^{-c_n L^{\partial \Omega^n}_t} \big| t < \tau_{\Omega^n} \right] \mathbb{P}^n_x(t < \tau_{\Omega^n}) + \mathbb{E}^n_x\left[ e^{-c_n L^{\partial \Omega^n}_t} \big| t \geq \tau_{\Omega^n} \right] \mathbb{P}^n_x(t \geq \tau_{\Omega^n})\\
= & \mathbb{P}^n_x(\tau_{\Omega^n} > t) + \mathbb{E}^n_x\left[ e^{-c_n L^{\partial \Omega^n}_t} \big| L^{\partial \Omega^n}_t >0 \right] \mathbb{P}^n_x(t \geq \tau_{\Omega^n}).
\end{align*}
From Proposition \ref{prop-L-w} and the fact that $\tau_{\Omega^n} \stackrel{law}{\to} \tau_\Omega$, we have that
\begin{equation}
\label{eq-proof618}
\mathbb{P}^n_x(\widehat{\zeta^{\Omega^n}} >t) \stackrel{w^\star}{\to} \mathbb{P}_x(\tau_{\Omega} > t) + \mathbb{E}_x\left[ e^{-c_\infty L^{\partial \Omega}_t} \big| L^{\partial \Omega}_t >0 \right] \mathbb{P}_x(t \geq \tau_{\Omega})
\end{equation}
and therefore we obtain:
\begin{itemize}
\item[i)] if $c_n \to c_0$, by considering that
\begin{align*}
\mathbb{E}_x\left[ e^{-c_0 L^{\partial \Omega}_t}\right] =  \mathbb{E}_x\left[ e^{-c_0 L^{\partial \Omega}_t} | L^{\partial \Omega^n}_t=0 \right] \mathbb{P}_x^n(L^{\partial \Omega^n}_t=0) +  \mathbb{E}_x\left[ e^{-c_0 L^{\partial \Omega}_t} | L^{\partial \Omega^n}_t>0 \right] \mathbb{P}_x^n(L^{\partial \Omega^n}_t>0)
\end{align*}
and the fact that $\mathbb{P}_x(L_t^{\partial \Omega}=0)=\mathbb{P}_x(\tau_\Omega> t)$, we have
\begin{align*}
\mathbb{E}_x\left[ e^{-c_0 L^{\partial \Omega}_t}\right] 
= & \mathbb{E}_x\left[ e^{-c_0 L^{\partial \Omega}_t} \big| L^{\partial \Omega}_t =0 \right] \mathbb{P}_x(\tau_{\Omega} > t) + \mathbb{E}_x\left[ e^{-c_0 L^{\partial \Omega}_t} \big| L^{\partial \Omega}_t >0 \right] \mathbb{P}_x(t \geq \tau_{\Omega})\\
= & \mathbb{P}_x(\tau_{\Omega} > t) + \mathbb{E}_x\left[ e^{-c_0 L^{\partial \Omega}_t} \big| L^{\partial \Omega}_t >0 \right] \mathbb{P}_x(t \geq \tau_{\Omega}).
\end{align*}
From \eqref{eq-proof618}, we obtain
\begin{align*}
\mathbb{P}^n_x(\widehat{\zeta^{\Omega^n}} >t) \stackrel{w^\star}{\to}  \mathbb{E}_x\left[ e^{-c_0 L^{\partial \Omega}_t}\right] = \mathbb{P}_x(\zeta^\Omega >t ), \quad \forall\, t\geq 0,\, (x \in \overline{\Omega})
\end{align*}
and thus, a.s. $T_{c_0} = \zeta^\Omega$;
\item[ii)] if $c_n \to 0$, $\mathbb{P}^n_x(\widehat{\zeta^{\Omega^n}} >t) \stackrel{w^\star}{\to} \mathbb{P}_x(\tau_{\Omega} > t) + \mathbb{P}_x(t \geq \tau_{\Omega}) = 1$ for all $t\geq 0$, $x \in \overline{\Omega}$ and 
\begin{align*}
\mathbb{P}_x(\omega \,:\, \lim_{n\to \infty} \widehat{\zeta^{\Omega^n}}(\omega) = \infty) =1 = \mathbb{P}_x(\omega \,:\, T_{0}(\omega) = \infty);
\end{align*}
\item[iii)] if $c_n \to \infty$, $\mathbb{P}^n_x(\widehat{\zeta^{\Omega^n}} >t) \stackrel{w^\star}{\to} \mathbb{P}_x(\tau_{\Omega} > t)$ for all $t > 0$, ($x \in \Omega$) and we conclude that a.s. $T_{\infty} = \tau_\Omega$.
\end{itemize}
Thus, we obtain that $\widehat{\zeta^{\Omega^n}} \stackrel{law}{\to} T_{c_\infty}$.

With this at hand, we now continue the proof. Since $\widehat{\zeta^{\Omega^n}} \leq \zeta^{\Omega^n_\varepsilon}$ with probability one, we have that $\mathbb{P}^n_x(\zeta^{\Omega^n_\varepsilon} >t) \geq \mathbb{P}^n_x(\widehat{\zeta^{\Omega^n} }>t)$. Thus, we have that $\widehat{\mathbf{P}^n_t} \mathbf{1} \leq \mathbf{P}^n_t \mathbf{1}$ and $\widehat{R^n_\lambda} \mathbf{1} \leq R_\lambda^n \mathbf{1}$ $m$-a.e. $x$. From the contraction property of $\mathbf{P}^n_t$ we have that 
\begin{align}
\label{contr-prop}
\|\widehat{R^n_\lambda} f \|_{L^2(\Omega^n_\varepsilon)} \leq \| R^n_\lambda f \|_{L^2(\Omega^n_\varepsilon)} \leq \lambda^{-1/2} \|f\|_{L^2(\Omega^n_\varepsilon)}
\end{align} 
for all measurable functions $f$. Strong convergence of resolvents in $L^2(\Omega^*)$ implies that, for $f \in L^2(\Omega^n_\varepsilon)$,
\begin{align*}
\lim_{n \to \infty} \| R^n_\lambda f \|_{L^2(\Omega^n_\varepsilon)} = \| R_\lambda f \|_{L^2(\Omega)}.
\end{align*}
Since $R^n_\lambda f$ is uniformly bounded, $\| \widehat{R^n_\lambda} f \|_{L^2(\Omega^n_\varepsilon)} < C$ for all $n$. From this and convergence a.e. we conclude that $\widehat{R^n_\lambda} f \to R_\lambda f$ weakly in $L^2(\Omega)$. Convergence of $\widehat{R^n_\lambda} f$ implies that 
\begin{align*}
\liminf_{n \to \infty} \| \widehat{R^n_\lambda} f \|_{L^2(\Omega^n_\varepsilon)} \geq \| R_\lambda f \|_{L^2(\Omega)}.
\end{align*}
Since we have that
\begin{align*}
\| R_\lambda f \|_{L^2(\Omega)} \leq \liminf_{n\to \infty} \| \widehat{R^n_\lambda} f \|_{L^2(\Omega^n_\varepsilon)} \leq  \liminf_{n \to \infty} \| R^n_\lambda f \|_{L^2(\Omega^n_\varepsilon)} = \lim_{n \to \infty} \| R^n_\lambda f \|_{L^2(\Omega^n_\varepsilon)} = \| R_\lambda f \|_{L^2(\Omega)}
\end{align*}
we conclude that
\begin{align}
\lim_{n \to \infty} \| R^n_\lambda f \|_{L^2(\Omega^n_\varepsilon)} = \lim_{n \to \infty} \| \widehat{R^n_\lambda} f \|_{L^2(\Omega^n_\varepsilon)}. \label{norm-vn}
\end{align}
Weak convergence of $\widehat{R^n_\lambda} f$ together with \eqref{norm-vn} says that $\widehat{R^n_\lambda} f \to R_\lambda f$ strongly in $L^2(\Omega)$.

Now we consider $f_n$ as in \eqref{111}. Since, from \eqref{contr-prop},
\begin{align*}
\| \widehat{R^n_\lambda}f_n - R_\lambda f \|_{L^2(\Omega)} \leq \| \widehat{R^n_\lambda}f_n - \widehat{R^n_\lambda} f \|_{L^2(\Omega)} + \| \widehat{R^n_\lambda} f  - R_\lambda f\|_{L^2(\Omega)}\\
\leq \lambda^{-1/2} \| f_n - f\|_{L^2(\Omega)} +  \| \widehat{R^n_\lambda} f  - R_\lambda f\|_{L^2(\Omega)}
\end{align*} 
we get the strong convergence of $\widehat{R^n_\lambda}f_n$ (or $\lambda \widehat{R^n_\lambda}f_n$).
\end{proof}

In light of the previous results, we can study the convergence of \eqref{semig-gen-A2} by considering the semigroup \eqref{semig-estim} and therefore the corresponding form. By taking into account \eqref{mean-A-pcaf} and \eqref{Rev-Corr}, we can study the form 
\begin{align}
\mathcal{E}^{\mu_{\widehat{A^n}}}_0(u,v) = \widetilde{a_n}(u,v) + \langle u,v\rangle_{\mu_{\widehat{A^n}}}, \quad u,v \in H^1(\mathbb{R}^2)  \cap L^2(\mu_{\widehat{A^n}})  \label{formEdominated}
\end{align}
associated with
\begin{equation}
\mathbb{E}^n_x \left[ f_n(\widetilde{B^{\nu(n),*}_t})\, e^{-c_n \int_0^t \mathbf{1}_{\partial \Omega^n}(\widetilde{B^{\nu(n),*}_s}) ds}  \right]
\label{F-K-sol-n}
\end{equation}
where $\widetilde{a_n}$ is defined in \eqref{a-tilde} and $\mu_{\widehat{A^n}}$ is supported on $\partial \Omega^n$.

\begin{proof}[Proof of Theorem \ref{thm-mu}]
From Theorem \ref{prop-exp} we have that
\begin{align}
\mathbb{E}^n_x \left[ \int_0^{\tau_{\Omega^n_\varepsilon}} e^{-\lambda t-\delta_n t} f_n(\widetilde{B^{\nu,*}_t}) \widehat{M^n_t} dt \right] \to R_\lambda f(x) = \mathbb{E}_x \left[ \int_0^\infty e^{-\lambda t- \delta_0 t} f(B^+_t) M_t dt \right]
\label{proof-res-M}
\end{align}
strongly in $L^2(\Omega)$ where the multiplicative functional $\widehat{M^n_t}$ and the killing functional are associated with the PCAF with Revuz measure $\widehat{\mu_n} = \delta_n + \mu_{\widehat{A^n}} + \infty_{(\Omega^n_\varepsilon)^c}$. Let us relate $R^n_\lambda$ to $\overline{\mu_n} = \delta_n + \mu_{A^n}$ in the same sense. From \eqref{proof-res-M} we have that, $\forall\, f \in C_0^+$,
\begin{align*}
\int f d\overline{\mu_n} \to \int f d\varrho \quad \Leftrightarrow \quad \int f d\widehat{\mu_n} \to \int f d\varrho.
\end{align*}
Indeed, the resolvents identify uniquely the multiplicative functionals. Thus we get that
\begin{align}
\overline{\mu_n} \stackrel{v}{\to} \widehat{\mu_\infty}.  
\label{w-con-mu}
\end{align}
This means equivalence between the corresponding additive functionals. Proposition \ref{prop-last} and Proposition \ref{prop-lastt} authorize us to study the asymptotic behaviour of
$$\widehat{A^n_t} = \frac{c_n}{2} (A^{n+}_t + A^{n-}_t)$$ 
and, by considering $\mu^+_n$ and $\mu^-_n$, with \eqref{ReU-def} and \eqref{Rev-Corr} in mind we get that
\begin{align*}
\lim_{\lambda \to \infty} \langle \lambda U^n_\lambda f, m^n_\varepsilon \rangle  = c_n \sigma_n \int_{\partial \Omega^n} f(x) \mathbf{1}_{\partial \Omega^n}(x) d \mathfrak{s} = \langle f, \mu_{\widehat{A^n}} \rangle .
\end{align*}
From \eqref{w-con-mu},  we get that 
\begin{equation*}
\overline{\mu_\infty} = \delta_0 + \left\lbrace 
\begin{array}{lll}
\displaystyle c_0\, \mu_\alpha\, \mathbf{1}_{\partial \Omega}+ \infty_{\overline{\Omega}^c}, & c_\infty=c_0 \in (0,\infty) & \textrm{(elastic kill)}, \\
\displaystyle 0 + \infty_{\overline{\Omega}^c}, & c_\infty = 0 & \textrm{(no kill)}, \\
\displaystyle \infty_{\Omega^c}, & c_\infty = \infty & \textrm{(kill on the boundary)}.
\end{array}
\right.
\end{equation*}

With \eqref{cond-nu} in mind, we write
\begin{equation}
\lim_{n\to\infty} c_n = \lim_{n\to \infty} \frac{ \sigma_n }{\textrm{tick}(\Sigma^n)} \int_{\partial \Omega^n} \mathbb{P}^n_{x} (B^{\nu(n),*}_t \in \Sigma^n)m^n_\varepsilon(dx) \label{cond-nu-2}
\end{equation}
where $w^n /\textrm{tick}(\Sigma^n) \to 1$ as $n\to \infty$ (tick($\Sigma^n$) is the thickness of the fiber).\\

{\bf iii)} If $c_n \to \infty$, then point iii) of the proof of Theorem \ref{prop-exp} holds and  is also in accord with \eqref{cond-nu-2}, $\tau_{\Omega^n_\varepsilon} \to \tau_\Omega$ faster than $B^{\nu(n),*}_t \to B^+_t$. The BM is killed on the boundary. On the other hand, if $\tau_{\Omega^n_\varepsilon} \to \tau_\Omega$, the lifetime on $\Omega$ of the limit process is exactly $\tau_\Omega$. Then for all $x \in \Omega$ we have that $\mathbb{P}_x(L^{\partial \Omega}_t = 0)=1$ for all $t < \tau_\Omega$ which means that 
\begin{equation*}
\forall\, x \in \Omega, \quad \mathbb{P}_x (e^{-c_n L^{\partial \Omega}_t} =1,  t < \tau_\Omega)=1, \quad \textrm{for all }  n\; (\textrm{for all  } c_n \geq 0).
\end{equation*}
This justifies the convention $\infty \cdot 0 = 0$. Moreover, if the lifetime is $\tau_\Omega$, from \eqref{maps-one}, it must be that $\zeta^n \to 0$, that is $c_n \to \infty$. Thus, $\tau_{\Omega^n_\varepsilon} \to \tau_\Omega$ if and only if condition \eqref{ISO} holds and $\partial \Omega$ becomes a  Dirichlet boundary as $n\to \infty$. We get that $\tau_{\Omega^n_\varepsilon} \to \tau_\Omega \Leftrightarrow c_n \to \infty$. 

{\bf ii)} If $c_n \to 0$, then according with \eqref{cond-nu-2} we say that $B^{\nu(n),*}_t \to B^+_t$ on $\Omega$ for $t \geq 0$ and therefore, $\zeta^{\Omega^n} \to \infty$. On the other hand, Remark \ref{rmk-life-boundary} and Remark \ref{rmk-life-domain} say that $c_n \to 0 \Rightarrow \widehat{\zeta^{\Omega^n}} \to \infty$. Since $\widehat{\zeta^{\Omega^n}} \leq \zeta^{\Omega^n_\varepsilon}$, $\widehat{\zeta^{\Omega^n}} \to \infty$ as well as the lifetime $\zeta^{\Omega^n_\varepsilon} \to \infty$. Moreover, $\zeta^{\Omega^n_\varepsilon} \to \infty \Rightarrow c_n \to 0$ by using Remark \ref{rmk-life-boundary}.

{\bf i)} Since $ii)$ and $iii)$ hold true, if $c_n \to c_0$ and $c_0 \neq 0$ or $c_0 \neq \infty$, then $\zeta^{\Omega^n_\varepsilon} \to T$ and $T \neq  \infty$ or $T \neq \tau_\Omega$. In particular, $T \in (\tau_\Omega, \infty)$ is a random variable depending on $c_0 \in (0, \infty)$. Thus, $\zeta^{\Omega^n_\varepsilon} \to T$ where $T$ depends on $c_0$. From Theorem \ref{prop-exp} we have that $T=\zeta^\Omega$ is an exponential random variable with parameter $c_0$. Thus, $c_n \to c_0 \Leftrightarrow \zeta^{\Omega^n_\varepsilon} \to \zeta^{\Omega}$. In this case (and also the previous as particular cases) we recover the results about the asymptotic of Robin problems on pre-fractal domains (\cite{CAP2} by considering the forms \eqref{formEdominated}).
\end{proof}

Let us focus on Remark \ref{rmk-Cap}. We observe that, as $n\to \infty$,
\begin{align*}
c_n \to c_0  >0 \quad \Leftrightarrow \quad \textrm{\normalfont Cap}_1(\mathbb{R}^2 \setminus \Omega) >0  \qquad \textrm{(transient case)}
\end{align*}
where $\Rightarrow$ immediately follows and $\Leftarrow$ is obtained from $iii)$ and $i)$. Moreover, from $ii)$,
\begin{align*}
c_n \to 0  \quad \Leftrightarrow \quad \textrm{\normalfont Cap}_1(\mathbb{R}^2 \setminus \Omega) = 0 \qquad \textrm{(recurrent case)}.
\end{align*}
Indeed, the process $X_t$ is transient iff Cap$_1(\mathbb{R}^2 \setminus \Omega)>0$ (\cite[Proposition 3.5.10]{chen-book}). A vanishing capacity can be related with the hitting distribution and in particular with the recurrence of the corresponding process. Let us consider the hitting distribution $H^\alpha_{\Lambda^c} \mathbf{1}_E (x) = \mathbb{E}_x[e^{-\alpha \tau_\Lambda} \mathbf{1}_E (B_{\tau_\Lambda})]$ where the expected value is taken under \eqref{hitting-desity}. For $\alpha >0$, $H^\alpha_{\Lambda^c} \mathbf{1}(x) = \mathbb{E}_x[e^{-\alpha \tau_\Lambda}]$. Since $\mathbb{R}^2 \setminus \Omega^n \downarrow \mathbb{R}^2 \setminus \Omega$ are sets of finite capacity, from \cite[Theorem 4.2.1]{FUK-book} we have that 
\begin{align*}
\textrm{Cap}_1(\mathbb{R}^2 \setminus \Omega^n) \to 0 \quad \textrm{iff} \quad H^1_{\mathbb{R}^2 \setminus \Omega^n} \mathbf{1}(x) \to 0 \;\; \textrm{q.e.}
\end{align*}
and thus, for $x \in \Omega^1_\varepsilon$, $\mathbb{P}_x(\tau_{\Omega^n} < \infty) \to 0$ (see also \eqref{def-life-sequence-2}).


\end{document}